\documentclass[final,onefignum,onetabnum]{siamonline190516}
\usepackage{braket,amsfonts}
\usepackage{xspace}
\usepackage{float}
\usepackage{etoolbox}

\input{def}
\newcommand{\diag}{\mathop{\mathrm{diag}}}
\definecolor{laurelgreen}{RGB}{169,186,157}
\definecolor{laurelgreen-solid}{RGB}{92, 112, 79}
\definecolor{blueyonder}{RGB}{80,114,167}
\pgfplotsset{
compat=1.11,
legend image code/.code={
\draw[mark repeat=2,mark phase=2]
plot coordinates {
  (0cm,0cm)
  (0.1cm,0cm)        %
  (0.2cm,0cm)         %
};%
}
}

\usepackage{lipsum}
\usepackage{graphicx}
\usepackage{epstopdf}
\usepackage{algorithmic}
\ifpdf
  \DeclareGraphicsExtensions{.eps,.pdf,.png,.jpg}
\else
  \DeclareGraphicsExtensions{.eps}
\fi

\usepackage{enumitem}
\setlist[enumerate]{leftmargin=.5in}
\setlist[itemize]{leftmargin=.5in}

\newcommand{\newremark}[2]{
  \theoremstyle{plain}
  \theoremheaderfont{\normalfont\itshape}
  \theorembodyfont{\normalfont\itshape}
  \theoremseparator{.}
  \theoremsymbol{}
  \newtheorem{#1}[theorem]{#2}
}
\newremark{remark}{Remark}
\allowdisplaybreaks

\def\addressncsu{Department of Mathematics, North Carolina State University,
  Raleigh, NC, USA}
\def\addressnyu{Courant Institute of Mathematical Sciences, New York
  University, New York, NY, USA}
\def\addressucm{Department of Applied Mathematics, University of California,
  Merced, CA, USA}

\author{
Alen Alexanderian\thanks{\addressncsu}
\and
Noemi Petra\thanks{\addressucm}
\and
Georg Stadler\thanks{\addressnyu} 
\and
Isaac Sunseri\footnotemark[2]
}

\title{Optimal design of large-scale Bayesian linear inverse problems
  under reducible model uncertainty: good to know what you don't
  know\thanks{Submitted to the editors \today.  \funding{Supported in
      part by US National Science Foundation DMS \#1723211, 
      \#1654311, and \#1745654.}}}

\headers{OED under reducible uncertainty}{A.~Alexanderian, N.~Petra, G.~Stadler, 
and I.~Sunseri}

\begin{document}
\maketitle

\begin{abstract}
We consider optimal design of infinite-dimensional Bayesian linear
inverse problems governed by partial differential equations that contain
secondary \emph{reducible} model uncertainties, in addition to the
uncertainty in the inversion parameters. By reducible uncertainties we
refer to parametric uncertainties that can be reduced through
parameter inference.  We seek experimental designs that minimize the
posterior uncertainty in the primary parameters, while accounting for
the uncertainty in secondary parameters.  We accomplish this by
deriving a marginalized A-optimality criterion and developing an
efficient computational approach for its optimization.  We illustrate
our approach for estimating an uncertain time-dependent
source in a contaminant transport model with an uncertain initial
state as secondary uncertainty.  Our results indicate that accounting
for additional model uncertainty in the experimental design process is
crucial.
\end{abstract}



\section{Introduction}
An inverse problem uses measurement data and a 
mathematical model to estimate a set of 
uncertain model parameters.
An experimental design specifies the strategy for collecting measurement data. 
For example, 
in inverse problems where measurement data are collected using sensors, 
an experimental design specifies the placement of
the sensors. This is the setting considered in the present
work.
Optimal experimental design (OED)~\cite{AtkinsonDonev92,Ucinski05} refers to the
task of determining an experimental setup such that the measurements
are most informative about the underlying parameters.  This is
particularly important in situations where experiments are costly or
time-consuming, and thus only a small number of measurements can be
collected.  
In addition to the parameters estimated by the inverse problem, the governing
mathematical models often involve simplifications, approximations, or modeling 
assumptions, resulting
in additional uncertainty. 
These additional
uncertainties must be taken into account in the experimental design process;
failing to do so could result in suboptimal
designs. 

We distinguish between two types of uncertainties: \emph{reducible}
and \emph{irreducible}~\cite{Smith13}. Reducible uncertainties, also
referred to as epistemic uncertainties, are those that can be reduced
through parameter inference. In contrast, irreducible uncertainties, also known as aleatoric
uncertainties, are inherent to the model and are impractical or
impossible to reduce through parameter inference. In this article, we
aim at computing optimal experimental designs in the presence of
\emph{reducible} model uncertainty.

In what follows, we consider the model
\begin{equation}\label{equ:model}
    \obs = \E(\iparm, \iparb) + \vec\eta,
\end{equation}
where $\obs$ is a vector of measurement data, $(\iparm, \iparb)$ a pair of
uncertain parameter vectors or functions, $\E$ a model that maps $(m, b)$ to
measurements, and $\vec\eta$ a random vector that models additive
measurement errors.  Herein, $\iparm$ is the parameter of primary interest,
which we seek to infer, and $\iparb$ represents additional uncertain
parameters. We assume $m$ and $b$ are elements of infinite-dimensional Hilbert
spaces.  Furthermore, we assume that the uncertainty in $\iparb$ is reducible.
Thus, we can formulate an inverse problem to estimate both $\iparm$ and
$\iparb$.  However, when designing experiments to solve the inverse problem,
our main interest is reducing the uncertainty in $\iparm$.  
We achieve this
by finding sensor placements that minimize the posterior
uncertainty in $\iparm$, while taking into account the uncertainty in $\iparb$.
This results in an OED problem in which we minimize the marginal
posterior uncertainty in $\iparm$.  

In this article, 
we focus on the case of a model that is linear in 
$\iparm$ and $\iparb$ and is of the form:
\begin{equation}\label{equ:forward_model}
    \E(\iparm, \iparb) = \F \iparm + \G\iparb. 
\end{equation}
Here, $\F$ and $\G$ are bounded linear transformations from suitably defined
Hilbert spaces to the space of measurement data. This models, for example,
a linear inverse problem with uncertain volume or boundary terms.
The mathematical foundations for Bayesian inversion and design of experiments
in this context are discussed in \cref{sec:bayes}.

Examples for secondary uncertainties
are initial conditions, boundary conditions that are introduced into a model due 
to the necessity to truncate a
computational domain, or unknown forcing or source terms in a real world
system that are only incorporated approximately in the mathematical model.
When designing experiments, failure to properly account for these secondary uncertainties 
may result in suboptimal experimental designs. For instance, not taking into account
secondary uncertainties in the mathematical model may result in sensors being
located close to uncertain sources, resulting in observations that can provide
biased information on the parameter of primary interest. If
one aims at finding designs that are optimal for both primary and secondary
uncertain parameters, the design is likely to be suboptimal for inference
of the primary parameter.

\paragraph{Related work}
In many inverse problems, one has model uncertainties in addition to
the inversion parameters.  A robust parameter inversion strategy must
account for such additional model uncertainties;
see~\cite{KaipioSomersalo05,KolehmainenTarvainenArridgeEtAl11,
  Aravkin12VanLeeuwe12,KaipioKolehmainen13,
  Nagel17,NicholsonPetraKaipio18,ConstantinescuBessacPetraEtAl20} for
a small sample of the literature addressing such issues.  This
work is about A-optimal experimental design for Bayesian linear
inverse problems governed by partial differential equations (PDEs)
with model uncertainties.  For a review of the literature on optimal
design of inverse problems governed by computationally intensive
models, we refer to~\cite{Alexanderian20}. Here, we mainly review
related work on optimal design of linear inverse problems.
The articles~\cite{HaberHoreshTenorio08,HaberMagnantLuceroEtAl12,
AlexanderianPetraStadlerEtAl14} present methods for large-scale ill-posed
linear inverse problems.  Specifically, the present article builds on
\cite{AlexanderianPetraStadlerEtAl14}, which focuses on A-optimal
experimental design of infinite-dimensional Bayesian linear inverse problems.

Recent work 
also considers A-optimal design of infinite-dimensional Bayesian linear inverse
problems with model uncertainties
\cite{KovalAlexanderianStadler20}. The key difference to
the present work is
that \cite{KovalAlexanderianStadler20}  considers
OED for inverse problems governed by models with \emph{irreducible} uncertainties and
formulates the OED problem as one of optimization under uncertainty.  In
contrast, in this work we consider OED under \emph{reducible}
model uncertainties and propose a formulation that
aims at minimizing the marginal posterior variance of 
the primary parameters. 
By combining primary and secondary uncertainties, the problem considered in
this work can formally be written as goal-oriented OED problem, as studied in
\cite{AttiaAlexanderianSaibaba18}.  However, taking a model uncertainty
perspective and considering infinite-dimensional primary and secondary
uncertain parameters require a tailored approach that distinguishes primary
and secondary uncertainties.

Other related efforts include~\cite{FohringHaber16,RuthottoChungChung18,
HermanAlexanderianSaibaba20}.  
In
\cite{FohringHaber16}, the authors present an adaptive A-optimal design
strategy for linear dynamical systems.
OED for linear inverse problems with linear equality and inequality
constraints is
addressed in \cite{RuthottoChungChung18}. This results in OED
with an effectively nonlinear inverse problem for which the  authors
propose an approach based on Bayes risk minimization.
In \cite{HermanAlexanderianSaibaba20}, the authors present an approach for
A-optimal design of infinite-dimensional Bayesian linear inverse problems using
ideas from randomized subspace iteration and reweighted $\ell_1$-minimization.

\paragraph{Contributions} This article makes the following contributions to 
the state-of-the-art in OED for 
large-scale linear inverse problems.
(i) We provide a mathematical formulation of OED under reducible model 
uncertainty and show how the OED problem can be reformulated to take advantage
of the often low dimensionality of the measurement space (see
\cref{sec:marginalized_Aoptimality}); in particular, our formulation eliminates
the need for trace estimation in the discretized parameter space. (ii) We
develop a scalable computational framework for solving the class of OED problems
under study (see \cref{sec:method}). Specifically, the computational 
complexity of   
our methods, in terms of the number of
PDE solves, does not grow with the dimension of the discretized primary and 
secondary parameters. 
(iii) We present illustrative
numerical results, in context of a contaminant transport inverse problem (see
\cref{sec:model} and \cref{sec:numerics}) where we seek to estimate an unknown
source term, but have additional uncertainty in the initial state.  Our
numerical experiments examine different aspects of our proposed framework, and
elucidate the importance of incorporating additional model uncertainties in the
OED problem.

\section{Bayesian inverse problems governed by models
with reducible uncertainties}\label{sec:bayes}
After introducing basic notation in \cref{sec:notation}, 
we present preliminaries regarding
Gaussian measures on Hilbert spaces in \cref{sec:marginals}. 
Next, we outline
the setup of Bayesian linear inverse problems with additional reducible model
uncertainties in infinite-dimensions (\cref{sec:bayes_setup}) as well as in
discretized form (\cref{sec:discretize}). We discuss 
basics on optimal design of such inverse problems in \cref{sec:OED_basic}.

\subsection{Notation}\label{sec:notation}
Herein we consider a probability space
$(\Omega,\mathfrak{A},\mathbb{P})$, where $\Omega$ is a sample space,
$\mathfrak{A}$ a sigma-algebra on $\Omega$, and $\mathbb{P}$ is a probability
measure.  Given a Hilbert space $\mathscr{X}$, we denote by
$\borel(\mathscr{X})$ the Borel sigma-algebra on $\mathscr{X}$. 
A Gaussian measure on $(\mathscr{X}, \borel(\mathscr{X}))$, with 
mean $\bar z \in \mathscr{X}$ and covariance operator $\mathcal{C}:\mathscr{X} \to \mathscr{X}$, 
is denoted by $\GM{\bar z}{\mathcal{C}}$.
We also recall that for a random variable $Z:(\Omega, \mathfrak{A}, \mathbb{P}) \to 
(\mathscr{X}, \borel(\mathscr{X}))$, its law is a Borel measure $\mathcal{L}_Z$ 
on $\mathscr{Y}$, that satisfies $\mathcal{L}_Z(A) = \mathbb{P}(Z \in A)$ for every
$A \in \borel(\mathscr{Y})$~\cite{Williams91}.
Also, for
a linear transformation $T:\mathscr{X} \to \mathscr{Y}$, where
$\mathscr{Y}$ is another Hilbert space, we denote the adjoint 
by $T^*$. 

\subsection{Marginals of Gaussian measures}\label{sec:marginals}
Here we discuss some preliminaries regarding Gaussian measures and 
Gaussian random variables taking
values in Hilbert spaces.  First we record the following known result about the
law of a linear transformation of a Hilbert space-valued Gaussian random
variable, which we prove for completeness.

\begin{lemma}\label{lem:Gaussian_LT}
Let $\mathscr{X}$ and
$\mathscr{Y}$ be infinite-dimensional Hilbert spaces. Suppose $Z:\Omega \to \mathscr{X}$ 
is a Gaussian random variable with law $\mu = \GM{\bar z}{\mathcal{C}}$. Consider the
random variable $Y = T Z$, where $T:\mathscr{X} \to \mathscr{Y}$ is a bounded linear
transformation. Then, $Y:(\Omega,\mathfrak{A}, \mathbb{P}) 
\to (\mathscr{Y}, \borel(\mathscr{Y}))$ is a Gaussian random variable with law $\nu = \GM{T \bar z}{T \mathcal{C} T^*}$.
\end{lemma}
\begin{proof}
Using~\cite[Proposition 1.18]{DaPrato06}, we know that the law of the random variable
$T:(\mathscr{X}, \borel(\mathscr{X}), \mu) \to (\mathscr{Y}, \borel(\mathscr{Y}))$ 
is given by $\mu \circ T^{-1} = \GM{T\bar z}{T \mathcal{C} T^*} = \nu$. 
To complete the proof we show $\mathcal{L}_Y = \nu$. Namely, 
for every $A \in \borel(\mathscr{Y})$,
\[
    \mathcal{L}_Y(A) = \mathbb{P}( Y \in A) = \mathbb{P}(TZ \in A) 
    = \mathbb{P}(Z \in T^{-1}(A)) = \mu(T^{-1}(A)) = \nu(A).
\]
\end{proof}
Consider a 
Hilbert space $\hilb = \hilb_1 \times \hilb_2$, where $\hilb_1$ and $\hilb_2$ 
are real, separable, infinite-dimensional Hilbert spaces with inner 
products $\ipA{\,\cdot}{\cdot}$ and
$\ipB{\,\cdot}{\cdot}$, respectively. An element $z \in \hilb$ is of the 
form $z = (z_1, z_2)$ with $z_1 \in \hilb_1$ and $z_2 \in \hilb_2$,
respectively.
We assume that $\hilb$ is equipped with the
natural inner product
\[
\ipH{x}{y} = \ipA{x_1}{y_1} + \ipB{x_2}{y_2}, \quad x, y \in \hilb.
\] 
Let
$Z:(\Omega, \mathcal{F}, \mathbb{P}) \to (\hilb, \mathcal{B}(\hilb), \mu)$ be a
Gaussian random variable with law $\mu = \GM{\bar{z}}{\mathcal{C}}$.  The
marginal laws of $Z$ can be defined analogously to the finite-dimensional
setting, as shown next. This shows that the familiar
marginalization results for Gaussian random variables 
remain meaningful in infinite dimensions.

We denote realizations of $Z$ by $z = (z_1, z_2) = (\Pi_1 z, \Pi_2 z) \in \hilb$, 
where $\Pi_1$ and $\Pi_2$ denote linear projection operators onto $\hilb_1$ and
$\hilb_2$, respectively.  
The following result concerns the law of $\Pi_i Z$, $i=1, 2$, i.e., 
marginal laws of $Z$. 
\begin{lemma}\label{lem:marginal}
$Z_i = \Pi_i Z$ has a Gaussian law $\mu_i$ with mean 
$\bar{z}_i = \Pi_i \bar{z}$ and covariance operator
  $\mathcal{C}_{ii}$, which satisfies
\begin{equation}\label{equ:Cii}
\ip{\mathcal{C}_{ii} u}{v}_i = \int_{\hilb_i} 
\ip{s - \bar{z}_i}{u}_i\ip{s - \bar{z}_i}{v}_i \, \mu_i(ds), \quad i = 1, 2,\:
\text{ for all } u,v\in \hilb_i.
\end{equation}
\end{lemma}
\begin{proof}
By \cref{lem:Gaussian_LT}, $\Pi_i Z$ has a Gaussian law $\mu_i =
\mathcal{N}(\Pi_i \bar{z}, \mathcal{C}_{ii})$ with 
$\mathcal{C}_{ii} = \Pi_i \mathcal{C} \Pi_i^*$. 
It remains to show 
that $\mathcal{C}_{ii}$ satisifes~\cref{equ:Cii}.
Without loss of generality,
we assume $\bar{z} \equiv 0$ and show the result for $i = 1$.
By definition of the covariance operator $\mathcal{C}$ of $\mu$, we 
have
$\ipH{\mathcal{C} a}{b} = \int_\hilb \ipH{z}{a}\ipH{z}{b} \, \mu(dz)$,
for $a, b \in \hilb$.
Therefore, for arbitrary $u, v \in \hilb_1$, we have
\[
    \ip{\mathcal{C}_{11}u}{v}_1
   = \ipH{\mathcal{C}\Pi_1^* u}{\Pi_1^* v}
   = \int_{\hilb} \ipH{z}{(u,0)}\ipH{z}{(v,0)} \mu(dz)
   = \int_{\hilb_1} \ip{s}{u}_1 \ip{s}{v}_1 \mu_1(ds).
\]
\end{proof}

In the present work, $\hilb_1 = L^2(\T)$ and
$\hilb_2 = L^2(\D)$ with $\T$ and $\D$ bounded open sets in $\R^{d_i}$
with $d_i \in \{1, 2, 3\}$, for $i = 1, 2$. 
In this case, realizations of $\Pi_1 Z$ and $\Pi_2 Z$ are square-integrable 
functions on $\T$ and $\D$, respectively.
Thus, we can also view $\Pi_i Z$ as a random field. Consider, e.g.,
$Z_1 = \Pi_1 Z$. This marginalized random field
has mean $\bar{z}_1(x)$ and the following covariance function (kernel):
\[
    c_{11}(x, y) \defeq \int_\Omega (Z_1(x, \omega) - \bar{z}_1(x))
                               (Z_1(y, \omega)-\bar{z}_1(y))\,\mathbb{P}(d\omega).
\]
As expected, the (marginal) covariance operator
$\mathcal{C}_{11}$ can be written as an integral 
operator with kernel $c_{11}$.
To show this, we use~\cref{equ:Cii} and again, for simplicity,
assume $\bar{z} \equiv 0$. 
Note that
\begin{align*}
\ip{\mathcal{C}_{11} u}{v}_1 =
\int_{\hilb_1} \ip{s}{u}_1 \ip{s}{v}_1 \mu_1(ds)
&= \int_\Omega \ip{Z_1(\omega)}{u}_1 \ip{Z_1(\omega)}{v}_1 \, \mathbb{P}(d\omega)
\\
&= 
\int_\Omega \int_\T \int_\T
  Z_1(x,\omega)Z_1(y,\omega) 
u(x) v(y)\, dx \,dy\,\mathbb{P}(d\omega)
\\
&=
\int_\T 
\left[\int_\T
\left(\int_\Omega 
Z_1(x,\omega)Z_1(y,\omega) \,\mathbb{P}(d\omega)\right)
v(y) \, dy \right]
\,  u(x) \,dx
\\
&= \int_\T 
\left[\int_\T
c_{11}(x, y) v(y) \, dy
\right]
u(x) \, dx,
\end{align*}
where we used Fubini's theorem to change the order of the integrals.
From this, we deduce 
\[
[\mathcal{C}_{11} v](\cdot) 
= \int_\T c_{11}(\cdot, y) v(y) \, dy.
\]

In finite dimensions, we recover the following well-known~\cite{Tong12} result,
which we prove here for completeness:
\begin{lemma}\label{lem:marginal_fd}
Consider a Gaussian random vector
\[
\vec{Z} = \begin{bmatrix} \vec{Z}_1 \\ \vec{Z}_2 
\end{bmatrix} \sim \mathcal{N}(\bar{\vec{z}}, \mat{C}) = 
\mathcal{N}\left( \begin{bmatrix} \bar{\vec{z}}_1 \\ \bar{\vec{z}}_2 \end{bmatrix},
\begin{bmatrix} \mat{C}_{11} & \mat{C}_{12} 
\\ \mat{C}_{21} & \mat{C}_{22} \end{bmatrix} \right),
\]
where $\vec{Z}_1$ and $\vec{Z}_2$ denote subsets of entries of $\vec{Z}$ and the
mean and covariance matrix are partitioned consistent with partitioning of $\vec{Z}$.
Then, the marginals of $\vec{Z}$ are 
Gaussian, with $\vec{Z}_1 \sim \mathcal{N}(\bar{\vec{z}}_1, \mat{C}_{11})$ 
and $\vec{Z}_2 \sim
\mathcal{N}(\bar{\vec{z}}_2, \mat{C}_{22})$. 
\end{lemma}
\begin{proof}
Note that $\vec{Z}_1 = \mat{P} \vec{Z}$ with 
$\mat{P} = \begin{bmatrix}\mat{I} & \mat{0}\end{bmatrix}$,
where $\mat{I}$ is the identity matrix of dimension equal to that of 
$\vec{Z}_1$ and $\vec{0}$ the zero matrix of the same size as $\vec{Z}_2$.
Thus, $\vec{Z}_1 \sim  \GM{\mat{P}\bar{\vec{z}}}{\mat{P}\mat{C}\mat{P}^\tran} 
= \GM{\bar{\vec{z}}_1}{\mat{C}_{11}}$. Showing the statement about the law of
$\vec{Z}_2$ is analogous. 
\end{proof}

\subsection{Bayesian inverse problem setup}\label{sec:bayes_setup}
We consider a Bayesian linear inverse problem for $\theta = (\iparm, \iparb)
\in \hilb = \hilb_1 \times \hilb_2$ and where the forward model is of 
the form \cref{equ:forward_model}.
We assume Gaussian priors for the primary and secondary parameters,
which we denote by $m$ and $b$, respectively, 
and for simplicity of the presentation assume no prior correlation between  
$m$ and $b$.
The presented framework can 
be modified to allow for prior correlations
between $m$ and $b$. Thus, 
the prior law of $(m, b)$ is the product measure
$\prior = \priorm \otimes \priorb$, with $\priorm$ and $\priorb$ 
each Gaussian measures on $\hilb_1$ and $\hilb_2$, i.e.,
$
   \priorm = \GM{\iparprm}{\iprcovm} 
$
and
$
   \priorb = \GM{\iparprb}{\iprcovb}. 
$ 
Note that $\prior = \GM{\iparpr}{\iprcov}$
with $\iparpr = (\iparprm, \iparprb)$ and
$\iprcov = \iprcovm \times \iprcovb$, where
\[
     (\iprcovm \times \iprcovb) (u_1, u_2) = (\iprcovm u_1, \iprcovb u_2), \quad 
     (u_1, u_2) \in \hilb.
\]
 
The inverse problem under study considers inference of 
$\iparm$ and $\iparb$ using measurement data $\obs \in \R^\Nd$ and the model
\begin{equation}\label{equ:model_lin}
\obs = \F \iparm + \G\iparb + \vec\eta.
\end{equation}
The measurement noise vector $\vec\eta$ is assumed to be independent of
$(\iparm, \iparb)$, and 
we assume a Gaussian noise model, $\vec\eta \sim \GM{\vec{0}}{\ncov}$.
Under these assumptions, the posterior is the Gaussian measure 
$\postm = \GM{\iparpost}{\ipostcov}$ with~\cite{Stuart10}
\begin{equation}\label{equ:post_infdim}
    \ipostcov^{-1} = \E^\adj\ncov^{-1}\E + \iprcov^{-1}, 
    \quad 
    \iparpost = \ipostcov(\E^\adj \ncov^{-1} \obs + \iprcov^{-1}\iparpr).
\end{equation} 
Note that $\E^*$ denotes the adjoint of the linear transformation $\E$.
Specifically, $\E^*$ satisfies $\E^*\obs = (\F^*\obs, \G^*\obs) \in \hilb$, for 
$\obs \in  \R^\Nd$.

\subsection{The discretized problem}\label{sec:discretize}
Let $\dparm$ and $\dparb$ be discretized versions of 
$\iparm$ and $\iparb$. Recall that we consider a parameter space
$\hilb$ of the form $\hilb = L^2(\T) \times L^2(\D)$.
The discretized parameter space is $\hilbn = \R^{\Nm} \times \R^{\Nb} 
\cong \R^{n}$, where $\Nm$ and $\Nb$ are the dimensions 
of the discretized parameters $\dparm$ and $\dparb$, respectively, 
and $n = \Nm + \Nb$. An element $\vec{u} \in \hilbn$, 
$\vec{u} = (\vec{u}_1, \vec{u}_2)$ with  
$\vec{u}_1 \in \R^{\Nm}$ and $\vec{u}_2 \in \R^{\Nb}$, can be represented as
$\vec{u} = 
[\begin{matrix} \vec{u}_1^\tran & \vec{u}_2^\tran\end{matrix}]^\tran$.
The discretized parameter space is endowed with the inner product
\[
\mip{\vec{u}}{\vec{v}} = \vec{u}_1^\tran \mat{M}_1 \vec{v}_1 
+ \vec{u}_2^\tran \mat{M}_2 \vec{v}_2
= \vec{u}^\tran \MM \vec{v}, \quad \vec{u}, \vec{v} \in \hilbn, 
\]
with 
$\MM = \begin{bsmallmatrix} \mat{M}_1 & \mat{0}\\ \mat{0} & \mat{M}_2\end{bsmallmatrix}$, 
and where  
the ``weight'' matrices $\mat{M}_1$ and $\mat{M}_2$ 
are defined based on the method used to discretize the
$L^2$-inner products on $L^2(\T)$ and $L^2(\D)$, respectively; 
see~\cref{sec:numerics} for examples. 
The discretized forward operator is defined by 
\[
   \EE \dpar = \begin{bmatrix} \mat{F} & \mat{G} \end{bmatrix}
   \begin{bmatrix} \dparm \\  \dparb \end{bmatrix} =  
\mat{F} \dparm + \mat{G} \dparb,
\]
where $\mat{F}$ and $\mat{G}$ are discretizations of $\F$ and $\G$ in
\cref{equ:model_lin}. The respective marginal priors are
$
\GM{\dparprm}{\prcovm}
$
and
$
\GM{\dparprb}{\prcovb},
$
and the prior covariance is 
$\prcov = \begin{bsmallmatrix} \prcovm & \mat{0} \\ \mat{0} & \prcovb
\end{bsmallmatrix}$. 
Using~\cref{equ:post_infdim}, the posterior covariance operator satisfies
\[
\postcov^{-1} = \begin{bmatrix}
    \prcovm^{-1} + \mat{F}^\adj \ncov^{-1} \mat{F} & \mat{F}^\adj \ncov^{-1} \mat{G} \\
    \mat{G}^\adj \ncov^{-1} \mat{F} & \prcovb^{-1} + \mat{G}^\adj \ncov^{-1} \mat{G}
\end{bmatrix}.
\]
Computing the inverse of the block matrix on the right is facilitated by the
well-known formula for the inverse of a such matrices~\cite[Theorem 2.1(ii)]{LuShiou02}.
Specifically, we can show that
the covariance operator of the 
marginal posterior law of $\dparm$ is given by 
\begin{equation}\label{equ:marginalized_post_cov}
\postcovm \!=\! 
  \big( \prcovm^{-1} +  \mat{F}^\adj \ncov^{-1} \mat{F}
         - \mat{F}^\adj \ncov^{-1} \mat{G} ( \prcovb^{-1} + \mat{G}^\adj \ncov^{-1} \mat{G})^{-1}
\mat{G}^\adj \ncov^{-1} \mat{F}\big)^{-1}.
\end{equation}
Note also that for 
\[
\mat{F}: (\R^{\Nm}, \ipg{\cdot}{\cdot}{\mat{M}_1}) \to (\R^\Nd, \ipg{\cdot}{\cdot}{\R^{\Nd}})
\quad \text{and}\quad
\mat{G}: (\R^{\Nb}, \ipg{\cdot}{\cdot}{\mat{M}_2}) \to (\R^\Nd, \ipg{\cdot}{\cdot}{\R^{\Nd}}),
\]
where $\ipg{\cdot}{\cdot}{\R^{\Nd}}$ denotes the Euclidean inner product on 
$\R^\Nd$, 
the respective adjoint operators are defined by 
(cf.~e.g.,~\cite{Bui-ThanhGhattasMartinEtAl13})
\begin{align}\label{eq:fd-adj-operators}
   \mat{F}^\adj = \mat{M}_1^{-1} \mat{F}^\tran \quad\text{and}\quad
   \mat{G}^\adj = \mat{M}_2^{-1} \mat{G}^\tran.
\end{align}

The optimal design approach we follow consists of minimizing the average posterior variance in 
$\dparm$ by minimizing the trace of the marginal posterior covariance operator
defined in \cref{equ:marginalized_post_cov}. We call the resulting OED criterion
the \emph{marginalized A-optimality criterion}.
In~\cref{sec:marginalized_Aoptimality}, we derive an alternative
expression for the marginal posterior covariance operator, which is
useful
in applications which only allow low or moderate dimensional measurements.

\subsection{Optimal experimental design}\label{sec:OED_basic}
We formulate the sensor placement problem using  
the approach in~\cite{HaberMagnantLuceroEtAl12,
AlexanderianPetraStadlerEtAl14}. 
We assume 
$\vec{x}_i$, $i = 1, \ldots, \Nd$, represent a fixed 
set of candidate sensor locations. The goal is to select an optimal
subset of these locations.
We assign a
non-negative weight $w_i \in \R$ to each $\vec{x}_i$, 
$i = 1, \ldots, \Nd$.
An experimental design is specified by the vector
$\vec{w} = [w_1, w_2, \ldots, w_\Nd]^\tran$.
As detailed in \cite{HaberMagnantLuceroEtAl12,AlexanderianPetraStadlerEtAl14},
binary weight vectors are desirable to decide whether or not to place
a sensor in each of the candidate locations.
However, solving an 
OED problem with binary weights is challenging due to its 
combinatorial complexity. Thus, as 
in~\cite{AlexanderianPetraStadlerEtAl14}, we
relax the problem by considering weights $w_i \in [0,1]$, $i  = 1,
\ldots, \Nd$. Binary weights are obtained using sparsifying
penalty functions, as discussed further in \cref{subsec:sparsity}. 
An alternative approach to obtaining binary weights, which can 
be suitable for some problems, is a greedy strategy; see \cref{subsec:greedy}.

The vector $\vec{w}$ is introduced into the Bayesian inverse problem through
the data likelihood~\cite{AlexanderianPetraStadlerEtAl14}.  
We 
assume uncorrelated measurements; i.e., the noise covariance is diagonal,
$\ncov = \mathrm{diag}(\sigma^2_1, \sigma^2_2, \ldots, \sigma^2_\Nd)$, with
$\sigma^2_j$ the noise level at the $j$th sensor. For 
$\vec{w} \in \R^\Nd$, we define the diagonal weight 
matrix $\W = \mathrm{diag}(w_1, w_2, \ldots, w_\Nd)$ and the 
matrix $\Wn$ as follows:
\begin{align}\label{eq:Wsigma}
\Wn \defeq 
\mathrm{diag}\Big( \frac{w_1}{\sigma^2_1}, \frac{w_2}{\sigma^2_2}, \ldots, \frac{w_{n_d}}{\sigma^2_{n_d}}\Big)=
\sum_{j = 1}^\Nd w_j \sigma_j^{-2} \vec{e}_j \vec{e}_j^\top,
\end{align}
where $\vec{e}_j$ is the $j$th coordinate vector in $\R^\Nd$.
The $\vec{w}$-dependent MAP estimator and posterior covariance operator are then 
given by~\cite{AlexanderianPetraStadlerEtAl14}
\begin{equation}\label{equ:weighted_posterior}
\dparmap(\vec{w}) = \postcov(\vec{w})
\big( \mat{E}^* \Wn\obs + \prcov^{-1} \dparpr\big)
\quad \text{and}\quad
\postcov(\vec{w}) = 
(\mat{E}^\adj \Wn \mat{E} + \prcov^{-1})^{-1}.
\end{equation} 

Optimal experimental design (OED) is the problem of finding a design that, within
constraints on the number of sensors allowed, minimizes the posterior
uncertainty in the estimated parameters. This is done by minimizing certain
design criteria that quantify the posterior uncertainty~\cite{Ucinski05,ChalonerVerdinelli95}. 
In this 
article, we use the A-optimal design criterion which is given by
$\trace\big[\postcov(\vec{w})\big]$; this criterion quantifies the average
posterior variance of the parameter $\dpar$. Using this approach for
\cref{equ:weighted_posterior}, the OED objective
is given by the sum of the average posterior variance
of the primary and secondary parameters.  The primary parameter being the main
focus of parameter estimation, we seek sensor placements that minimize the
uncertainty in the primary parameter, while being aware of the uncertainty in
the secondary parameters.  This is done by 
finding designs that minimize the average posterior
variance of the primary parameters, quantified according to the corresponding
marginalized posterior distribution. We call such
designs marginalized A-optimal designs, which are the subject of
\cref{sec:marginalized_Aoptimality}.

Note that ignoring the uncertainty in the secondary parameter  
and fixing $\vec b$ to some nominal value $\vec b_0$, results in  
the affine forward model
$    \mat{E}_{\scriptscriptstyle0} \vec{m} = \mat{F}\vec{m} + \mat{G}\vec{b}_0$.
In this case, the posterior law of $\dparm$ is
$\GM{\dparmapm}{\postcovm}$ with
\begin{equation}\label{equ:weighted_posterior_fixedb}
\begin{aligned}
\dparmapm(\vec{w}) &= \postcovm(\vec{w})
\big( \mat{F}^* \Wn(\obs - \mat{G}\vec{b}_0) + \prcovm^{-1} \dparprm\big)
\quad \text{and}\\
\postcovm(\vec{w}) &=
(\mat{F}^\adj \Wn \mat{F} + \prcovm^{-1})^{-1},
\end{aligned}
\end{equation}
and an A-optimal design $\vec w$ is
one that minimizes the classical A-optimality criterion
\begin{equation}\label{equ:classical_Aoptimal}
    \psi(\vec w) \defeq \trace\big[(\mat{F}^*\Wn\mat{F} +
       \prcovm^{-1})^{-1}\big].
\end{equation}
Notice that the optimal design does not depend on the choice of $\vec b_0$.
More importantly, such an optimal design is completely unaware of 
the uncertainty in $\vec b$.

\section{Marginalized Bayesian A-optimality}\label{sec:marginalized_Aoptimality}
In this section, we present our formulation of the marginalized A-optimality
criterion.  We first
derive a reformulation of the marginalized posterior covariance that
facilitates an efficient computational procedure for computing marginalized
A-optimal designs; see \cref{sec:reformulated_posterior}. 
Then, we present the definition of the marginalized A-optimality criterion,
in~\cref{sec:criterion},
and prove its convexity.  Finally, the formulation of the optimization problem for
finding  marginalized A-optimal designs is discussed in~\cref{sec:optimization}.

\subsection{Alternative form of the posterior}
\label{sec:reformulated_posterior}
Computing optimal designs based on the marginalized posterior covariance
operator \cref{equ:marginalized_post_cov} entails traces of
operators defined on the discretized parameter spaces.  The corresponding
expressions also include inverses of operators of dimensions $\Nm$ and $\Nb$;
see~\cref{equ:marginalized_post_cov}.
The discretized parameter dimensions are typically large and depend
on the computational grids used for discretization.  In
many large scale inverse problems, the dimension $\Nd$ of the 
measurement vector $\obs$ is considerably smaller than the dimension of the
discretized uncertain parameters. Also, in our approach, this measurement
dimension is fixed a priori.  Here we derive an alternative expression for the
posterior covariance operator~\cref{equ:weighted_posterior} that facilitates
exploiting this problem structure.  In particular, this allows reformulating
the marginalized A-optimality criterion in terms of an operator defined on the
measurement space, which can then be computed directly (see \cref{sec:method}).
This is in contrast to previous works such
as~\cite{HaberHoreshTenorio08,HaberMagnantLuceroEtAl12,
AlexanderianPetraStadlerEtAl14,FohringHaber16,HermanAlexanderianSaibaba20} that
use randomized trace estimation (in the discretized parameter space) to compute
the OED objective.

\begin{theorem}\label{thm:posterior_cov}
The following relation holds.
\begin{equation}\label{equ:posterior_cov}
(\mat{E}^\adj \Wn \mat{E} + \prcov^{-1})^{-1}  = \prcov - \prcov\mat{E}^\adj(\mat{I} + \Wn \mat{E} \prcov\mat{E}^\adj)^{-1}\Wn \mat{E} \prcov.
\end{equation}
\end{theorem}
\begin{proof}
First, we need to show that $\mat{I} + \Wn \mat{E} \prcov\mat{E}^\adj$ is
invertible. To do this, we show that $\Wn \mat{E} \prcov\mat{E}^\adj$ has
non-negative eigenvalues.
Note that $\prcov = \prcov^\adj = \MM^{-1} \prcov^\tran \MM$.
Moreover, we have that $\mat{E}^\adj = \MM^{-1}
\mat{E}^\tran$. Thus, we have  $(\mat{E} \prcov\mat{E}^\adj)^\tran =
(\mat{E}^\adj)^\tran \prcov^\tran \mat{E}^\tran = \mat{E} \MM^{-1} \prcov^\tran
\MM \mat{E}^\adj = \mat{E} \prcov \mat{E}^\adj$.  That is, $\mat{E}
\prcov\mat{E}^\adj$ is symmetric; it is also clearly positive semidefinite.  

To show that $\Wn \mat{E} \prcov\mat{E}^\adj$ has non-negative
eigenvalues, we recall a basic result from linear algebra: if
$\mat{A}$ and $\mat{B}$ are two square matrices, $\mat{A}\mat{B}$ and
$\mat{B}\mat{A}$ have the same eigenvalues; see e.g.,~\cite[page
249]{Ortega87}.  Applying this result with $\mat{A} = \Wn^{1/2} \mat{E}
\prcov\mat{E}^\adj$ and $\mat{B} = \Wn^{1/2}$, we have that $\Wn^{1/2} \mat{E}
\prcov\mat{E}^\adj \Wn^{1/2}$ and  $\Wn \mat{E} \prcov\mat{E}^\adj$ have the  same
eigenvalues. Therefore, since  
$\Wn^{1/2} \mat{E} \prcov\mat{E}^\adj \Wn^{1/2}$ is symmetric positive semidefinite, 
it follows that $\Wn \mat{E} \prcov\mat{E}^\adj$ has non-negative eigenvalues.
This implies that $\mat{I} + \Wn \mat{E} \prcov\mat{E}^\adj$ is invertible. 
The relation~\cref{equ:posterior_cov} is now seen as follows:
\[
   \begin{aligned}
   (&\mat{E}^\adj \Wn \mat{E} + \prcov^{-1})(\prcov - \prcov\mat{E}^\adj(\mat{I} + \Wn \mat{E} \prcov\mat{E}^\adj)^{-1}\Wn \mat{E} \prcov)\\
   &=
   \mat{E}^\adj \Wn \mat{E}\prcov
   - \mat{E}^\adj \Wn \mat{E}\prcov\mat{E}^\adj(\mat{I} + \Wn \mat{E} \prcov\mat{E}^\adj)^{-1}\Wn \mat{E} \prcov 
   + \mat{I} - \mat{E}^\adj(\mat{I} + \Wn \mat{E} \prcov\mat{E}^\adj)^{-1}\Wn \mat{E} \prcov\\
  &= \mat{I} + \mat{E}^\adj \Wn \mat{E}\prcov - \mat{E}^\adj ( \Wn \mat{E}\prcov\mat{E}^\adj + \mat{I}) (\mat{I} + \Wn \mat{E} \prcov\mat{E}^\adj)^{-1}\Wn \mat{E}
     \prcov\\
  &= \mat{I} +  \mat{E}^\adj \Wn \mat{E}\prcov - \mat{E}^\adj\Wn \mat{E} \prcov = \mat{I}.
   \end{aligned}
\]
\end{proof}
Notice that this result is well known in the case $\Wn = \ncov^{-1}$. 
The challenge here is to account for the 
possibility of a singular $\Wn$. 
Note that the expression in the left hand side of \cref{equ:posterior_cov} 
involves the inverse of an $n \times n$ matrix, where $n = \Nm +
\Nb$, whereas
the expression on the right hand side involves the inverse of an $\Nd \times \Nd$ matrix. 
It is also worth noting that the proof of \cref{thm:posterior_cov} can be
simplified by the use of the Sherman--Morrison--Woodbury
formula. 
Above, we chose to 
present a direct linear algebra argument instead, for clarity.

We introduce the following notations, which will be used in the remainder of this article.
\begin{equation}\label{equ:Q}
   \Q(\vec{w}) := \left( \mat{I} + \Wn\FG\right)^{-1}\Wn,
   \quad \text{where} \quad
   \FG := \mat{F} \prcovm \mat{F}^\adj + \mat{G}\prcovb\mat{G}^\adj.
\end{equation}

Next, we present tractable representations for the posterior 
mean and covariance operator 
in a
(discretized) Bayesian linear inverse problem, as formulated in 
\cref{sec:discretize}.
Recall that 
the primary parameter
is $\dparm$ and the secondary parameter is $\dparb$.

\begin{theorem}\label{thm:posterior_blocks}
The posterior law of $\begin{bmatrix} \dparm \\ \dparb \end{bmatrix}$
is 
$\GM{\begin{bmatrix} \dparmapm \\ \dparmapb \end{bmatrix}}
   {\begin{bmatrix}
    \postcovm(\vec{w})
    &
    \postcovmb(\vec{w})\\
    \postcovmb^*(\vec{w})
    &
    \postcovb(\vec{w})
 \end{bmatrix}}$,
where
\begin{equation}\label{eq:post_mb_blocks}
\begin{alignedat}{2}
&
\postcovm(\vec{w})  && = \prcovm - \prcovm \mat{F}^\adj \Q(\vec{w}) \mat{F} \prcovm,
\\ 
&
\postcovb(\vec{w})  && = \prcovb - \prcovb \mat{G}^\adj \Q(\vec{w}) \mat{G} \prcovb,
\\
&
\postcovmb(\vec{w}) && = -\prcovm \mat{F}^\adj \Q(\vec{w}) \mat{G} \prcovb,\\
&
\dparmapm(\vec{w}) && = \postcovm(\vec{w})(\mat{F}^* \Wn\obs + 
                     \prcovm^{-1}\dparprm)+ 
                     \postcovmb(\vec{w})(\mat{G}^* \Wn \obs + 
                     \prcovb^{-1} \dparprb),
\\
&
\dparmapb(\vec{w}) && = \postcovb(\vec{w})(\mat{G}^* \Wn\obs + 
                     \prcovb^{-1}\dparprb)+
                     \postcovmb^*(\vec{w})(\mat{F}^* \Wn \obs + 
                     \prcovm^{-1} \dparprm).
\end{alignedat}
\end{equation}
\end{theorem}

\begin{proof}
Recall that the discretized forward operator $\mat{E}$ can be represented in 
a block matrix form $\mat{E} = \begin{bmatrix} \mat{F} & \mat{G} \end{bmatrix}$. 
Using this and the 
expression for $\postcov$ given in \cref{thm:posterior_cov}, 
we obtain
\begin{equation}\label{equ:postcov_blocks}
\begin{aligned}
\postcov(\vec{w}) &= 
\prcov
 -
\prcov
\begin{bmatrix} \mat{F}^\adj 
\\ 
\mat{G}^\adj\end{bmatrix} 
 \left( \mat{I} + \Wn \begin{bmatrix} \mat{F} & \mat{G}\end{bmatrix} 
\prcov
\begin{bmatrix} \mat{F}^\adj 
\\ 
\mat{G}^\adj\end{bmatrix}\right)^{-1}
\Wn\begin{bmatrix} \mat{F} & \mat{G}\end{bmatrix}
\prcov
\\
&=
\prcov
 -
\prcov
 \begin{bmatrix} \mat{F}^\adj \\ \mat{G}^\adj\end{bmatrix}
 \left( \mat{I} + \Wn 
 (\mat{F} \prcovm \mat{F}^\adj + \mat{G}\prcovb\mat{G}^\adj)\right)^{-1}
\Wn\begin{bmatrix} \mat{F} & \mat{G}\end{bmatrix}
\prcov
\\
&= 
 \begin{bmatrix}
    \prcovm & \mat{0} \\ \mat{0} & \prcovb
 \end{bmatrix}
 -
 \begin{bmatrix}
    \prcovm & \mat{0} \\ \mat{0} & \prcovb
 \end{bmatrix}
 \begin{bmatrix} \mat{F}^\adj \\ \mat{G}^\adj\end{bmatrix}
\Q(\vec{w})
\begin{bmatrix} \mat{F} & \mat{G}\end{bmatrix}
\begin{bmatrix}
    \prcovm & \mat{0} \\ \mat{0} & \prcovb
 \end{bmatrix}
\\
&= 
\begin{bmatrix}
    \prcovm - \prcovm \mat{F}^\adj \Q(\vec{w}) \mat{F} \prcovm 
    &
    -\prcovm \mat{F}^\adj \Q(\vec{w}) \mat{G} \prcovb\\
    -\prcovb \mat{G}^\adj \Q(\vec{w}) \mat{F} \prcovm 
    &
    \prcovb - \prcovb \mat{G}^\adj \Q(\vec{w}) \mat{G} \prcovb
 \end{bmatrix}.
\end{aligned}
\end{equation}
This establishes the representation of the 
posterior covariance operator.
The expressions for $\dparmapm(\vec{w})$ and $\dparmapb(\vec{w})$
can be obtained 
using \cref{equ:weighted_posterior} and \cref{equ:postcov_blocks}.
\end{proof}

Using \cref{lem:marginal_fd} in conjunction with \cref{thm:posterior_blocks},
the marginal posterior laws of $\dparm$ and $\dparb$ are given by
$\GM{\dparmapm(\vec{w})}{\postcovm(\vec{w})}$ and
$\GM{\dparmapb(\vec{w})}{\postcovb(\vec{w})}$, respectively.  Since
we target
the primary parameter $\dparm$, we focus on the corresponding marginal
posterior law $\GM{\dparmapm(\vec{w})}{\postcovm(\vec{w})}$. The marginal
covariance operator $\postcovm(\vec{w})$ will be used to define the marginal
A-optimality criterion (see below). Also, note that the expression for 
$\dparmapm$ in \cref{eq:post_mb_blocks} is the sum of two terms:
the first is the familiar expression for the posterior mean if
$\dparb$ was fixed to $\dparb = \vec 0$; the second reflects the impact
of the uncertainty in $\dparb$. 

\subsection{The marginalized A-optimality criterion}
\label{sec:criterion}
The marginalized A-optimal design (\mOED) criterion is given by 
\begin{equation}\label{equ:criterion}
  \Phi(\vec{w}) := \trace(\postcovm(\vec{w})) = \trace(\prcovm) - \trace(\prcovm \mat{F}^\adj \Q(\vec{w}) \mat{F} \prcovm). 
\end{equation}
Next, we show the convexity of the \mOED{} objective.
Before proving this, we consider a slightly more general result.
Below, $\Snm$ denotes the cone of self-adjoint and positive definite operators on $\R^n$ 
equipped with the weighted inner product 
$\mip{\cdot}{\cdot}$.

\begin{theorem}\label{thm:convexity_gen}
Let the function $f:\R^\Ns_{\geq 0} \to \R$ be given by 
\[f(\vec{w}) = \trace(\mat{R} \postcov(\vec{w})\mat{R}^\adj),\]
where
$\mat{R}$ is an $n \times n$ matrix and $\mat{R}^*$ denotes its adjoint
with respect to $\mip{\cdot}{\cdot}$. Then, the function $f$ is convex. 
\end{theorem}
\begin{proof}
Let $\mat{A}(\vec{w}) =  \postcov(\vec{w})^{-1}$, and 
note that $\mat{A}(\vec{w}) \in \Snm$ for all $\vec{w} \in \R^\Ns_{\geq 0}$.
First we show the function 
$G(\mat{A}$) = $\trace(\mat{R}\mat{A}^{-1}\mat{R}^\adj)$ 
is convex on $\Snm$.
Consider the restriction of $G$ to a line, $\mat{S} + t\mat{B}$, 
where $\mat{S} \in \Snm$ and  
$\mat{B}$ is self-adjoint; we consider values of $t$ for which 
$\mat{S} + t\mat{B} \in \Snm$. Let $\mat{U}\mat\Lambda\mat{U}^\adj$ be the spectral
decomposition of $\mat{V} = \mat{S}^{-1/2}\mat{B}\mat{S}^{-1/2}$;
here $\mat\Lambda$ is a diagonal matrix with the eigenvalues $\{\lambda_i\}_{i=1}^n$ 
of $\mat{V}$ on 
its diagonal and $\mat{U}$ is a matrix with the corresponding eigenvectors
$\{\vec{u}_i\}_{i = 1}^n$ as its columns.
Letting $\mat{L} = \mat{S}^{-1/2}\mat{R}^\adj$,
we note
\[
\begin{aligned}
  G(\mat{S}+t\mat{B})
&= \trace(\mat{R}\mat{S}^{-1/2}(\mat{I} + t\mat{S}^{-1/2}\mat{B}\mat{S}^{-1/2})^{-1}\mat{S}^{-1/2}\mat{R}^\adj) \\
&= \trace(\mat{L}\mat{L}^\adj(\mat{I} + t\mat{V})^{-1}) 
= \sum_{i=1}^n \mip{\mat{L}\mat{L}^\adj (\mat{I}+t\mat{V})^{-1} \vec{u}_i}{\vec{u}_i} 
= \sum_{i=1}^n (1+t\lambda_i)^{-1} \mip{\mat{L}^\adj \vec{u}_i}{\mat{L}^\adj \vec{u}_i}.
\end{aligned}
\]
Thus, $G(\mat{S}+t\mat{B})$ is a
linear combination of convex functions with non-negative coefficients, 
$\mip{\mat{L}^\adj \vec{u}_i}{\mat{L}^\adj \vec{u}_i} \geq 0$, and is thus convex. 
This shows that $G$ is convex on $\Snm$. 
It remains to show that $f(\vec{w})$ = $G(\mat{A}(\vec{w})$) is convex. Recall that
$\mat{A}(\vec{w})$ = $\prcov^{-1} +
\mat{E}^\adj \Wn \mat{E}$; thus $\mat{A}$ is affine in $\vec{w}$ 
and therefore, for $\alpha \in [0,1]$, 
\begin{multline*}
 f(\alpha \vec{w} + (1- \alpha) \vec{v}) = G(\mat{A}(\alpha \vec{w} + (1- \alpha) \vec{v})) 
 = G(\alpha\mat{A}(\vec{w}) + (1-\alpha)\mat{A}(\vec{v}))\\ 
 \leq \alpha G(\mat{A}(\vec{w})) + (1-\alpha)G(\mat{A}(\vec{v})) 
 = \alpha f(\vec{w}) + (1-\alpha) f(\vec{v}). 
 \end{multline*}
\end{proof}

\begin{corollary}\label{cor:convexity}
The function $\Phi: \R^\Nd_{\geq 0} \to \R$, 
defined in \cref{equ:criterion}, is convex.
\end{corollary}
\begin{proof}
Using~\cref{equ:postcov_blocks}, we can write
$\Phi(\vec{w})$ as 
\[
  \Phi(\vec{w}) = \trace(\mat{R} \mat{\postcov}(\vec{w}) \mat{R}^\adj)
\quad \text{with } \quad
\mat{R} = \begin{bmatrix} \mat{I} & \mat{0}\\ \mat{0} & \mat{0}\end{bmatrix}.
\]
Thus, the convexity of $\Phi(\vec{w})$ can be concluded 
from Proposition~\ref{thm:convexity_gen}.
\end{proof}

Consider the marginalized A-optimality criterion 
$\Phi(\vec w)$ in \cref{equ:criterion}.
Since the prior covariance operator is independent of $\vec w$,
minimizing $\Phi(\vec{w})$ is equivalent to minimizing
\begin{equation}\label{equ:obj}
\Psi(\vec{w}) := 
-\trace(\mat{F}\prcovm^2\mat{F}^\adj \Q(\vec{w})).
\end{equation}
This is the objective function we use in finding a marginalized A-optimal 
design.
Henceforth, we refer to this objective function as
the \mOED{} objective or the \mOED{} criterion.

\subsection{Computing optimal designs}
\label{sec:optimization}
Here we describe the optimization problem for computing \mOEDs.
The ultimate goal is to find a binary optimal design vector that
minimizes the \mOED{} objective $\Psi$, defined
in~\cref{equ:obj}.  That is, 
letting $\mathcal{X} = \{0, 1\}^\Nd$,
we would like to solve
\begin{equation}\label{equ:binary_optim_problem}
    \min_{\vec{w} \in \mathcal{X}} \Psi(\vec{w}), \quad
    \text{s.t.} \sum_{i=1}^\Nd w_i = N,
\end{equation}
where $N$ is a desired number of sensors.  However, as mentioned
above, solving such a binary optimization problem can be intractable due
to its combinatorial complexity.  One possibility to find an
approximate solution to this problem is via a greedy procedure, i.e.,
place  sensors one-by-one. This method does not require
derivatives of the objective with respect to weights. Greedy approaches
result, in general, in suboptimal solutions, which, in practice,
are often quite good. Computational details of this
approach are discussed in \cref{subsec:greedy}. 
We also compare, in \cref{sec:comparison}, the performance of the
greedy approach against the approach described next.

As an alternative to the greedy approach, one can consider a
relaxation of the problem and allow for design weights 
in the interval $[0, 1]$.  Binary weights are then obtained using
sparsifying penalty functions. Specifically, we consider an
optimization problem of the form
\begin{equation}\label{equ:optim_problem}
\min_{\vec{w} \in \mathcal{W}} \Psi(\vec{w}) + \gamma P(\vec{w}),
\end{equation}
where $\mathcal{W} = [0, 1]^\Nd$, $\Psi(\vec w)$ is the \mOED{}
objective, $\gamma > 0$ is a penalty parameter, and $P(\vec w)$ is a penalty
function. Minimization of \eqref{equ:optim_problem} usually requires
gradients of the objective. Key computational aspects are discussed in
the next section where we outline computational
methods for tackling the \mOED~problem.

\section{Computational methods}\label{sec:method}
In this section, we present a computational framework for computing~\mOEDs.

\subsection{Efficient computation of \mOED{} objective and its gradient}
\label{subsec:oedobs-and-grad}
Consider the objective function $\Psi(\vec w)$ defined
in~\cref{equ:obj}.  We note that the argument of the trace
in~\cref{equ:obj} is an operator defined on $\R^{\Nd \times \Nd}$,
where $\Nd$ is the number of candidate sensor locations (i.e., the
dimension of the measurement vector). This objective function can be
computed as follows:
\begin{equation}\label{equ:obj_compute}
\Psi(\vec{w}) 
   = 
   -\sum_{i=1}^\Nd {\vec{e}_i}^\tran{\FPF \Q(\vec{w}) \vec{e}_i}
   =-\sum_{i=1}^\Nd \vec{e}_i^\tran\FPF\vec{q}_i, 
\quad \text{where } 
\FPF = \mat{F}\prcovm^2\mat{F}^\adj,
\end{equation}
$\vec{q}_i =
\Q(\vec{w}) \vec{e}_i$ with $\Q(\vec{w})$ given in \cref{equ:Q}, 
and $\vec{e}_i$ is the $i$th standard basis vector in $\R^\Nd$, $i = 
1, \ldots, \Nd$. Note that
\begin{equation}\label{equ:qi}
\vec{q}_i = (\mat{I} + \Wn\FG)^{-1}\Wn \vec{e}_i
          = \sigma_i^{-2} w_i (\mat{I} + \Wn\FG)^{-1} \vec{e}_i.
\end{equation}
To derive the expression for the gradient of 
$\Psi$, we first need the following derivative:
\[
\begin{aligned}
\frac{\partial}{\partial w_j} \Q(\vec{w}) 
&= -\sigma_j^{-2}\left( \mat{I} + \Wn\FG\right)^{-1}(\vec{e}_j\vec{e}_j^\tran)\FG
\left( \mat{I} + \Wn\FG\right)^{-1}\Wn
+\sigma_j^{-2}\left( \mat{I} + \Wn\FG\right)^{-1}\vec{e}_j\vec{e}_j^\tran.
\end{aligned}
\]
Thus, 
\[
\begin{aligned}
\frac{\partial\Psi}{\partial w_j} &= -\frac{\partial}{\partial w_j}
\trace(\Q(\vec{w}) \FPF )\\ 
&=\trace\left[\sigma_j^{-2}\left( \mat{I} + \Wn\FG\right)^{-1}\vec{e}_j\vec{e}_j^\tran\FG
\left( \mat{I} + \Wn\FG\right)^{-1}\Wn\FPF\right]
-\trace\left[\sigma_j^{-2} \left( \mat{I} + \Wn\FG\right)^{-1} \vec{e}_j\vec{e}_j^\tran
\FPF\right]
\\
&= 
\sigma_j^{-2}\vec{e}_j^\tran\FG(\mat{I} + \Wn\FG)^{-1} \Wn\FPF
(\mat{I} + \Wn\FG)^{-1} \vec{e}_j 
- \sigma_j^{-2}\vec{e}_j^\tran \FPF (\mat{I} + \Wn\FG)^{-1}\vec{e}_j\\
&= 
\sum_{i = 1}^\Nd 
w_i \sigma_i^{-2}\sigma_j^{-2}\vec{e}_j^\tran\FG(\mat{I} + \Wn\FG)^{-1} \vec{e}_i\vec{e}_i^\tran
\FPF
(\mat{I} + \Wn\FG)^{-1} \vec{e}_j 
- \sigma_j^{-2}\vec{e}_j^\tran \FPF (\mat{I} + \Wn\FG)^{-1}\vec{e}_j,
\end{aligned}
\]
where we have used the cyclic property of the trace and the definition of 
$\Wn$ in \cref{eq:Wsigma}.
Letting $\vec{y}_i = (\mat{I} + \Wn\FG)^{-1} \vec{e}_i$, $i = 1, \ldots, \Nd$,
and substituting in the above expression,
leads to
\begin{equation}\label{equ:grad}
\frac{\partial\Psi}{\partial w_j} = 
\sum_{i = 1}^\Nd
w_i \sigma_i^{-2}\sigma_j^{-2}(\vec{e}_j^\tran \FG\vec{y}_i) \vec{e}_i^\tran 
\FPF
\vec{y}_j
- \sigma_j^{-2}\vec{e}_j^\tran 
\FPF\vec{y}_j,  
\quad 
j = 1, \ldots, \Nd.
\end{equation}
Note that the vectors $\vec{q}_i$ in~\cref{equ:qi} and 
vectors $\vec{y}_i$ in the definition of the gradient are related 
according to $\vec{q}_i = w_i \sigma_i^{-2} \vec{y}_i$, 
$i = 1, \ldots, \Nd$.

The matrices $\FG$ and $\FPF$ in \cref{equ:qi,equ:grad}
are of size $\Nd \times \Nd$. As mentioned previously, in many cases,
the measurement dimension $\Nd$ is considerably smaller than the dimension of
the discretized primary and secondary parameters.  This case
typically arises in
inverse problems governed by PDEs, where the
dimension of the discretized parameters grow upon grid refinements,
while the measurement dimension $\Nd$ is fixed a priori. 

The matrices $\FG$ and $\FPF$ can be built in a
precomputation step, as outlined in~\cref{alg:precompute}. The computational
cost to build $\FG$ and $\FPF$ is $3\Nd$ forward and $2\Nd$ adjoint PDE solves.
Once the matrices $\FG$ and $\FPF$ are computed, the OED objective and gradient
evaluation can be performed without further PDE solves and require only 
linear algebra operations; see \cref{alg:aopt}.
The cost of evaluating
the objective function is dominated by the cost of steps 1--3, which amount
to computing $\mat{Y} = (\mat{I} + \Wn\FG)^{-1}$; this can be done in
$\mathcal{O}(n_\text{d}^3)$ arithmetic operations, by precomputing an LU
factorization of $\mat{I} + \Wn\FG$ and then performing triangular solves to
compute columns of $\mat{Y}$. We also need the matrix-matrix product
$\mat{D}\mat{Y}$ (see step 5 of~\cref{alg:aopt}), which requires an additional $\mathcal{O}(n_d^3)$ 
operations.  The additional effort in computing the gradient is dominated
by one matrix-matrix product, $\mat{C}\mat{Y}$, amounting to
$\mathcal{O}(n_\text{d}^3)$ arithmetic operations.

\begin{algorithm}[t!]
\caption{Computing matrices $\mat{C}$ in~\cref{equ:Q} and
  $\mat{D}$ in \cref{equ:obj_compute} needed for \mOED~objective and gradient evaluation.}
\begin{algorithmic}[1]\label{alg:precompute}

\FOR {$i = 1$ to $\Nd$}
\STATE Compute $\vec{a}_i = \prcovm \mat{F}^\adj\vec{e}_i$
\STATE Compute $\vec{d}_i = \mat{F}\prcovm \vec{a}_i$
\hfill\COMMENT{columns of $\FPF = \mat{F}\prcovm^2\mat{F}^\adj$}
\STATE Compute $\vec{c}_i = \mat{F}\vec{a}_i + \mat{G}\prcovb\mat{G}^\adj \vec{e}_i$
\hfill\COMMENT{columns of $\FG = \mat{F} \prcovm \mat{F}^\adj + \mat{G}\prcovb\mat{G}^\adj$}
\ENDFOR
\STATE Build $\FG = 
[\begin{matrix} \vec{c}_1 & \cdots & \vec{c}_\Nd\end{matrix}]$
and 
$\FPF = [\begin{matrix} \vec{d}_1 & \cdots & \vec{d}_\Nd\end{matrix}]$
\end{algorithmic}
\end{algorithm}

\renewcommand{\algorithmicrequire}{\textbf{Input:}}
\renewcommand{\algorithmicensure}{\textbf{Output:}}
\begin{algorithm}[h!]
\caption{Computing $\Psi(\vec{w})$ and its gradient
  $\nabla\Psi(\vec{w})$.}
\begin{algorithmic}[1]\label{alg:aopt}
\REQUIRE Design vector $\vec{w}$.
\ENSURE  $\Psi = \Psi(\vec{w})$ and $\nabla\Psi = \nabla\Psi(\vec{w})$
\STATE \texttt{/* evaluation of the objective function */}
\FOR {$i = 1$ to $\Nd$}
\STATE Solve the system 
$(\mat{I} + \Wn\FG)\vec{y}_i = \vec{e}_i$
\ENDFOR
\STATE Compute $\displaystyle\Psi = -\sum_{i=1}^\Nd 
w_i \sigma_i^{-2}  \vec{e}_i^\tran \FPF\mat{Y} \vec{e}_i$
\hfill
\COMMENT{$\mat{Y} = [\begin{matrix} \vec{y}_1 & \vec{y}_2 & \cdots & \vec{y}_\Nd
\end{matrix}]$}
\STATE \texttt{/* evaluation of the gradient */}
\FOR {$j = 1$ to $\Nd$}
\STATE Compute
$\displaystyle
\frac{\partial\Psi}{\partial w_j} =
\sum_{i = 1}^\Nd w_i \sigma_i^{-2}\sigma_j^{-2}(\vec{e}_j^\tran \FG \mat{Y} \vec{e}_i) 
                     (\vec{e}_i^\tran \FPF \mat{Y} \vec{e}_j) 
                     - \sigma_j^{-2} \vec{e}_j^\tran \FPF \mat{Y} \vec{e}_j$
\ENDFOR
\end{algorithmic}
\end{algorithm}

\subsection{Sparsity control}\label{subsec:sparsity}
Here we discuss several options for choosing 
the penalty function $P(\vec w)$ in~\cref{equ:optim_problem}.
A straightforward choice for $P(\vec w)$ 
is the $\ell_1$-norm of $\vec w$; see
e.g.,~\cite{HaberHoreshTenorio08,HaberMagnantLuceroEtAl12}.  As is well-known,
the $\ell_1$-penalty promotes sparsity, but not necessarily a binary structure,
in the computed design vectors.  Another option is to solve a sequence of
optimization problems where penalty functions approximating $\ell_0$-``norm'' 
(the number of nonzero elements in a vector)
are used. An example is the so-called regularized $\ell_0$-sparsification
approach proposed in~\cite{AlexanderianPetraStadlerEtAl14}; in this approach, 
which we use in the present work, a
continuation approach is used, and a sequence of optimization problems, with
non-convex penalty functions approaching the $\ell_0$-norm, are solved.  A
related approach is the use of reweighted $\ell_1$-minimization, as done
in~\cite{HermanAlexanderianSaibaba20}.  Solving optimization problems with
continuous weights, combined with a suitable penalty method, enables the use of
powerful gradient-based optimization algorithms to explore the set of
admissible designs. The effectiveness of such approaches in obtaining optimal
sensor placements has been demonstrated in a number of previous works; see
e.g.,~\cite{HaberHoreshTenorio08,HaberMagnantLuceroEtAl12,AlexanderianPetraStadlerEtAl14,HermanAlexanderianSaibaba20}.

\subsection{Greedy sensor placement}\label{subsec:greedy}
An alternative approach for
finding sparse \mOEDs{} is to use a greedy strategy. 
Greedy approaches have been used successfully in many sensor
placement applications to obtain designs that, while suboptimal, provide near
optimal performance; see
e.g.,~\cite{KrauseSinghGuestrin08,ChamonRibeiro17,ShulkindHoreshAvron18,
Jagalur-MohanMarzouk20}. 
In a greedy approach, we place sensors one at a time: in
each step, we select a sensor that provides the largest decrease in the 
design criterion. A greedy approach can be attractive due to its simplicity and the
fact that  it does not require the gradient of the design criterion. However, the
computational complexity of greedy sensor placement, in terms of function
evaluations, scales with the number of candidate sensor locations and the number 
of the sensors in the optimal design. Note that the computational cost, in terms of function evaluations, of 
finding a greedy sensor placement (in its most basic form) with 
$K$ sensors is 
\begin{equation}\label{equ:greedy_cost}
     C(K, \Nd) = K \Nd - (K-1)K/2.
\end{equation}

\section{Model problem setup}\label{sec:model}
To illustrate our approach for computing optimal designs under
reducible uncertainty, we consider a linear inverse problem governed
by a time-dependent advection-diffusion equation with two sources of
uncertainty: the parameter of primary interest is a time-dependent
scalar-valued function $\iparm=\iparm(t)$, which models the time amplitude
of a source entering on the right hand side of the equation. The
second uncertain parameter is the spatially distributed initial
condition $\iparb=\iparb(\vec x)$. Specifically, we consider:
\begin{subequations}\label{eq:ad}
\begin{alignat}{2}
    u_t - \kappa\Delta u + \vec{v}\cdot\nabla u &= \delta(\vec{x})m(t) & \quad&\text{in
    }\D\times \DT,\label{eq:ad1} \\
    u(\cdot, 0) &= \iparb(\vec{x})  &&\text{in } \D ,\label{eq:ad2} \\
    \kappa\nabla u\cdot \vec{n} &= 0 &&\text{on } \partial\D \times
    \DT. \label{eq:ad3}
\end{alignat}
\end{subequations}
Here, $\D$ is a bounded open set in $\R^2$, the time interval
$\DT=(0,T)$, where $T>0$ is a final time, $\kappa > 0$ is the
diffusion coefficient, and $\vec{v}$ is a given velocity field.
Note that the solution $u(\vec{x},t)$, which can be interpreted as concentration, depends affinely on $m$
and $b$. In
our numerical experiments, $\kappa = 0.001$ and $\D$ is a unit square with two cutouts as shown in 
\cref{fig:velocity_prior}~(left). If \eqref{eq:ad1} models the flow of
a contaminant in a region, the cutouts could represent
buildings, for instance. The velocity field $\vec{v}$ (shown in \cref{fig:velocity_prior}) 
is obtained by
solving Navier-Stokes equations with no-outflow boundary conditions
and non-zero tangential boundary conditions as
in~\cite{AlexanderianPetraStadlerEtAl14}.
The function $\delta$ in the source term is given
by a mollified delta-function:
\begin{equation}\label{eq:source}
    \delta(\vec{x}) = \left(\frac{1}{2\pi L} e^{-\frac{1}{2L^2}
                   \| \vec{x} - \vec{x}_0\|^2}\right),
\end{equation}
where the ``correlation length'' $L$ is $0.05$ in our experiments, and
$\vec{x}_0=(0.5,0.35)$ as indicated by the red dot in 
\cref{fig:velocity_prior}~(left).

\subsection{Parameter-to-observable map}  The parameter-to-observable map
maps the time evolution of the right hand side amplitude, $\iparm\in
L^2(\T)$ and the initial condition $\iparb \in L^2(\D)$ to point
measurements of the solution of the advection-diffusion equation
\cref{eq:ad}.
To write it in the form
\cref{equ:forward_model}, we define the continuous linear operators
$\Sm$ and $\Sb$ as follows: $\Sm$ maps $m$ to the PDE solution $u$,
with $b=0$, and $\Sb$ maps $b$ to the PDE solution $u$, with
$m=0$.
Then, the solution to the initial-boundary value problem~\cref{eq:ad}
can be written as $u = \Sm \iparm + \Sb\iparb$; see~\cite[p.152]{Troltzsch10}.
Next, let $\B$ be a linear observation operator that extracts the
values of $u(\vec{x}, t)$ on a set of sensor locations $\{\vec{x}_1,
\vec{x_2}, \dots, \vec{x}_{\Nd}\}\in\D$, and takes an average of $u$
over the  time interval $[0.95,0.99]$.
Then $\F = \B \Sm $ and $\G = \B \Sb $ map the primary inference parameter
$m$ and the additional uncertain parameter $b$ to measurement $\obs
\in \R^{\Nd}$:
\begin{equation}\label{equ:param-to-obs}
\F :\iparm(t)
   \,
   \stackrel{\Sm}{\longmapsto}
   \,
   u(\vec{x}, t)
   \,
   \stackrel{\B}{\longmapsto}
   \,
   \obs, \qquad 
   \G :\iparb(\vec{x})
   \,
   \stackrel{\Sb}{\longmapsto}
   \,
   u(\vec{x}, t)
   \,
   \stackrel{\B}{\longmapsto}
   \,
   \obs. 
\end{equation}
The
corresponding discrete parameter-to-observable maps $\FF$ and $\GG$
are obtained through discretization using, for instance, finite elements.

Computations of derivatives of an objective that involves the
parameter-to-observable map requires the adjoint operators $\F^{\adj}$
and $\G^{\adj}$. These can be derived using  the
formal Lagrangian method, resulting in the following adjoint
equations \cite{Troltzsch10}.  Given a vector of observations $\obs
\in \R^{\Nd}$,
we first solve the adjoint equation
(see~\cite{AkcelikBirosDraganescuEtAl05,AlexanderianPetraStadlerEtAl14})
for the  adjoint variable $p = p(\vec{x}, t)$
\begin{subequations}
\begin{alignat}{2}\label{eq:ad:adj}
    -p_t - \nabla \cdot (p \vec{v}) - \kappa\Delta p  &= - \B^{\adj} \obs &\quad&\text{ in
    }\D\times \DT,\\
    p(\cdot, T) &= 0  &&\text{ in } \D,  \\
    (\vec{v}p+\kappa\nabla p)\cdot \vec{n} &=  0 &&\text{ on }
    \partial\D\times \DT,
\end{alignat}
\end{subequations}
and obtain the action of the adjoint operators as $\F^*\obs =
-\int_{\D} f(\vec{x}) p(\vec{x}, \cdot) d\vec{x}$ and $\G^*\obs =
-p(\cdot, 0)$.

\begin{figure}[bt]\centering
\includegraphics[width=.36\textwidth]{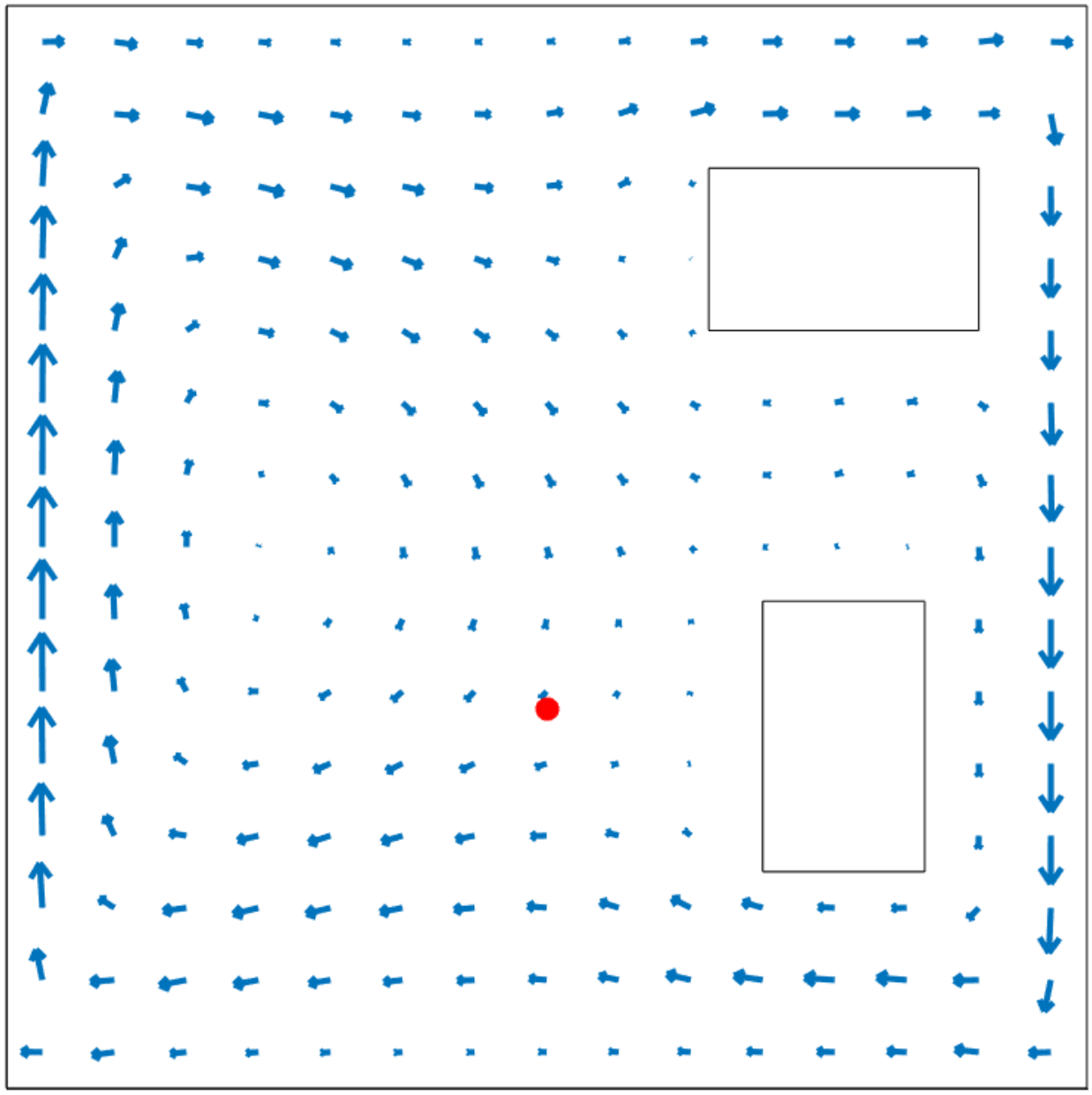}\hspace{.5cm}
\includegraphics[width=.5\textwidth]{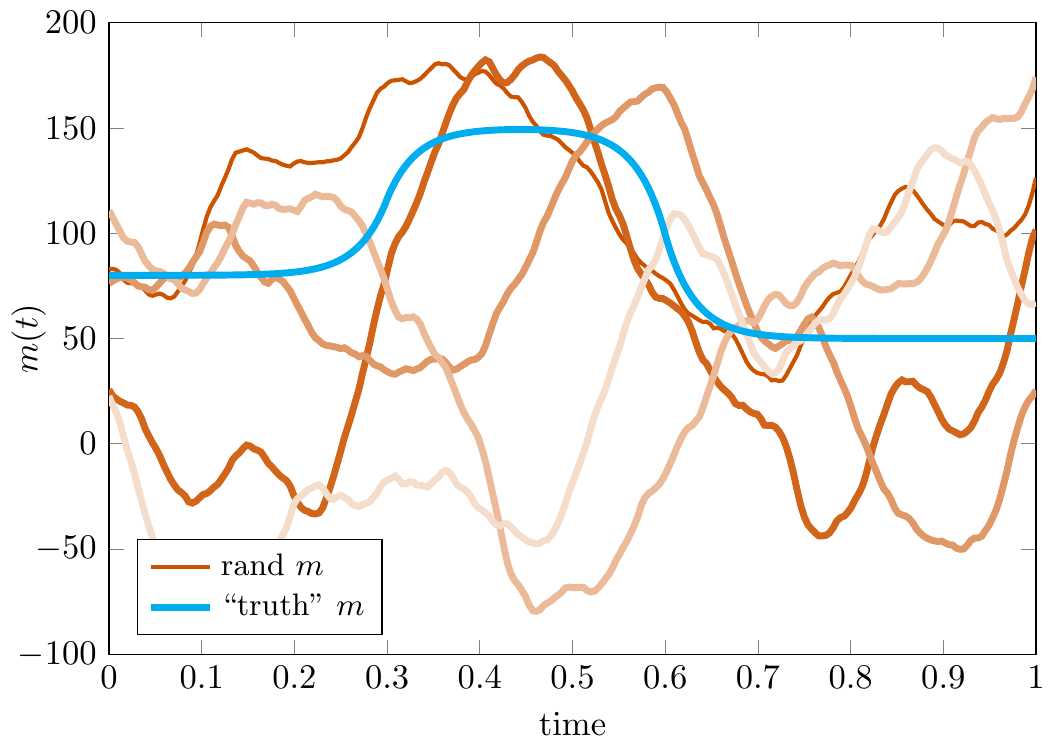}
\caption{Left: Sketch of domain $\D$ and velocity field $\vec{v}$ in
\eqref{eq:ad}. The red dot indicates the location
$\vec{x}_0=(0.5,0.35)$ where the source term \eqref{eq:source} is centered.
Right: the ``truth'' source term
$\iparm$ and five samples from the prior distribution of $\iparm$ shown in cyan
and various shades of orange, respectively.} 
\label{fig:velocity_prior}
\end{figure}

\subsection{Prior laws of $\iparm$ and $\iparb$} To complete the
definition of the Bayesian inverse problem, we specify the prior laws
for $\iparm$ and $\iparb$. We assume both to be Gaussian random
fields, and thus it is sufficient to specify the mean and covariance
operator. For the primary parameter $\iparm$, which is a function of
time only, we choose the mean to be the constant function
$\iparprm\equiv 65$, and specify the covariance operator $\iprcovm$ according 
to
\[
    [\iprcovm z](t) = \int_\T c(s, t) z(t) \, dt, \quad z \in L^2(\T),
\]
where we chose the Mat\'ern-3/2 covariance kernel
\begin{equation}\label{equ:matern_cov} 
   c(s, t)= \sigma^{2}
            \left(1+{\frac {{\sqrt {3}}|s-t|}{\ell}}\right)
            \exp\left(-{\frac {{\sqrt {3}}|s-t|}{\ell }}\right).
\end{equation}
This covariance function ensures that draws from the prior law of $m$ are
(almost surely) contintinuously differentiable; see,
e.g.,~\cite{HandcockStein93,WilliamsRasmussen06,LindgrenRueLindstrom11}.  In
our numerical experiments, we use the parameters $\sigma = 80$ and $\ell =
0.17$ in~\eqref{equ:matern_cov}. Samples from the resulting
distribution are shown in \cref{fig:velocity_prior}~(right).

The realizations of the secondary parameter $\iparb$ are functions
defined over the spatial domain $\D$. For the distribution of $\iparb$ we
choose a Gaussian with mean
$\iparprb\equiv 50$, and a Laplacian-like covariance
operator of the form $(-\epsilon\Delta + \alpha
I)^{-2}$ \cite{Stuart10}, with $\epsilon=4.5\times10^{-3}$ and
$\alpha=2.2\times10^{-1}$. We equip the Laplace operator with 
homogeneous Robin boundary conditions with constant coefficient.
We do this to mitigate undesired boundary
effects that can arise when PDE operators are used to define
covariance operators \cite{RoininenHuttunenLasanen14,DaonStadler18}.

\subsection{Discretization}
We discretize the forward problem using linear finite elements on
triangular meshes in space and use the implicit Euler method in time. 
This guides the discretization of the primary and secondary uncertainties
$m$ and $b$. Specifically, the discretized uncertain source terms is
the vector $\vec{m}$ whose entries are the values of $m$ at the time-steps
used by the forward solver. We discretize the $L^2(\T)$ inner product
using quadrature. That is, for $f, g \in L^2(\T)$,
\[
    \ip{f}{g}_1 = \int_\T f(t) g(t) \,dt 
    \approx \sum_{j=1}^{\Nm} \nu_j f(t_j) g(t_j) = \vec{f}^\tran \mat{M}_1 \vec{g}
    =: \ip{\vec{f}}{\vec{g}}_{\mathrm{M}_1},
\]
where $\{ \nu_j \}_{j=1}^{\Nm}$ are quadrature weights,
$\vec{f}$ and $\vec{g}$ are vectors (in $\R^{\Nm}$) of function values at the time-steps, and 
$\mat{M}_1 = \diag(\nu_1, \nu_2, \ldots, \nu_{\Nm})$. In the present work, 
we use the composite trapezoid rule to discretize the $L^2(\T)$ inner product.

The uncertain initial state $b$ is discretized using finite element Lagrange nodal 
basis functions, $\varphi_1(\vec{x})$, \ldots, $\varphi_{\Nb}(\vec{x})$. This
leads to the discretization 
$
    b(\vec{x}) \approx b_h(\vec{x}) = \sum_{j=1}^{\Nb} b_i \varphi_i(\vec{x}).
$
The discretized initial state is given by the vector $\vec{b}$ of
finite-element coefficients. This finite element method is also used
to discretize the PDE operator $(-\epsilon\Delta + \alpha I)$, which
is the square root of the covariance operator of the distribution of
$\vec{b}$. The covariance operator is thus defined as the square of
the finite element operator, corresponding to a mixed discretization
of the 4th-order covariance operator \cite{Bui-ThanhGhattasMartinEtAl13}.
Also, note that the discretized $L^2(\D)$-inner product is 
given by 
$
    \ip{\vec{u}}{\vec{v}}_{\mathrm{M}_2} = \vec{u}^\tran \mat{M}_2 \vec{v},
\text{ for }
\vec{u}, \vec{v} \in \R^{\Nb},
$
where $\mat{M}_2$ is the finite-element mass matrix.

In the numerical experiments below, we use a discretization with $\Nm = 257$
time steps and $\Nb = 1{,}529$ spatial degrees of freedom.  The ``truth'' primary
parameter $\iparm$ is shown in \cref{fig:velocity_prior}~(right), and the ``truth''
secondary parameter $\iparb$ is given by a random draw from the prior law of
$\iparb$, depicted in \cref{fig:forward_sol}~(top left).
For
computing solutions for the inverse problem, we synthesize data using
``truth'' parameters $b$ and $m$, and add Gaussian noise with 
standard deviation $\sigma_\text{noise} = 0.25$ to each data
point. That is, we assume $\ncov = \sigma_\text{noise}^2 \mat{I}$, with 
$\sigma_\text{noise} = 0.25$. Notice that the sensor measurements obtained from the model 
range approximately in the interval
$[51, 54]$; see e.g.,~\cref{fig:forward_sol}~(top right). Thus,
a noise standard deviation of $0.25$ is significant compared to the variations
of model output at the sensors.

\subsection{Illustrating the impact of the secondary uncertainty}\label{sec:primary_secondary}
To depict the impact of the secondary uncertainty on the solution of the forward 
problem, in
\cref{fig:forward_sol} 
we show snapshots of
the solution of the state equation.
Here, we use two
random draws from the prior distribution of $\iparb$, i.e., the
secondary uncertainty, as initial conditions. Recall that the initial
condition used for the first row is also used as ``truth'' secondary parameter.
For the primary uncertainty, the time evolution of the right hand
side source, the ``truth'' parameter (see
\cref{fig:velocity_prior}~(right)) is used.
Note that even at the final snapshot, around which measurements are
taken for inference, distinct
differences caused by the different initial conditions are visible.
This indicates that the uncertainty in the initial state cannot be ignored.

\def \pos {0.5\columnwidth}
\addtolength\abovecaptionskip{-15pt}
\begin{figure}[ht]\centering
  \begin{tikzpicture}
    \node (1) at (0*\pos-0.53*\pos, 0*\pos+1.1*\pos){\includegraphics[width=.24\textwidth]{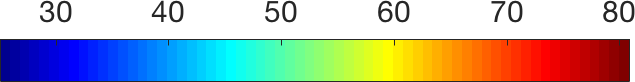}};
    \node (2) at (0.5*\pos, 0*\pos+1.1*\pos){\includegraphics[width=.24\textwidth]{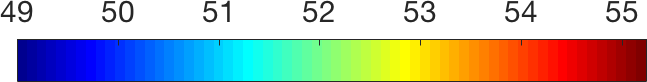}};
    \node (3) at (0*\pos-0.53*\pos, 1.0*\pos-0.2*\pos){\includegraphics[width=.24\textwidth]{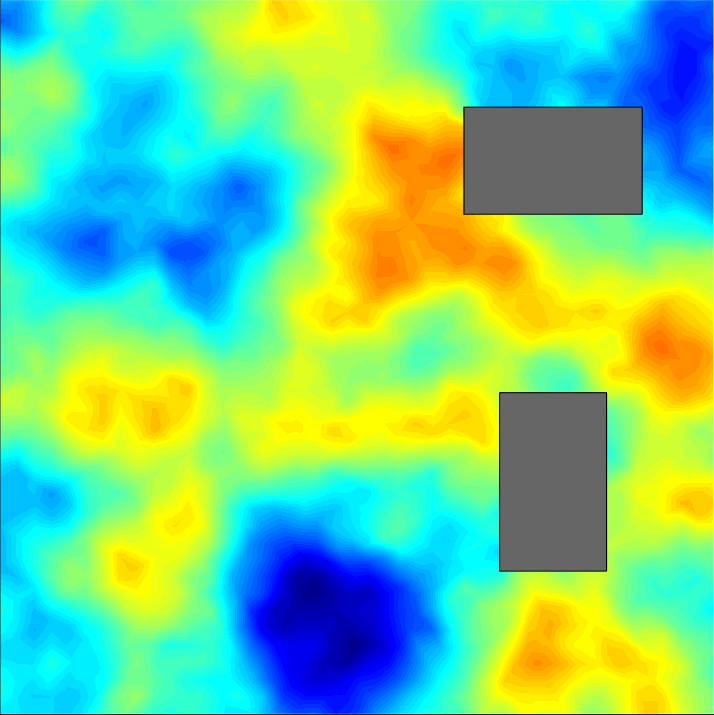}};
    \node (4) at (0.0*\pos, 1.0*\pos-0.2*\pos){\includegraphics[width=.24\textwidth]{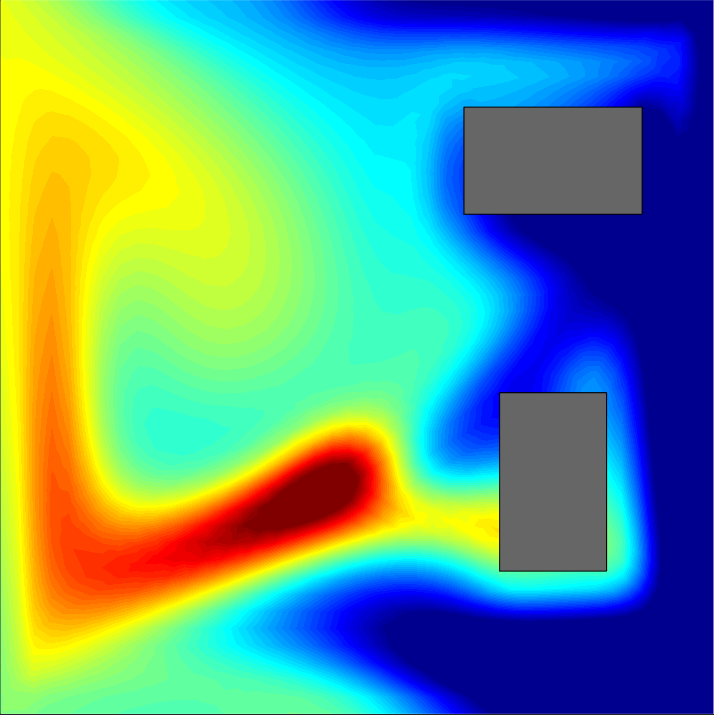}};
    \node (5) at (0.5*\pos, 1.0*\pos-0.2*\pos){\includegraphics[width=.24\textwidth]{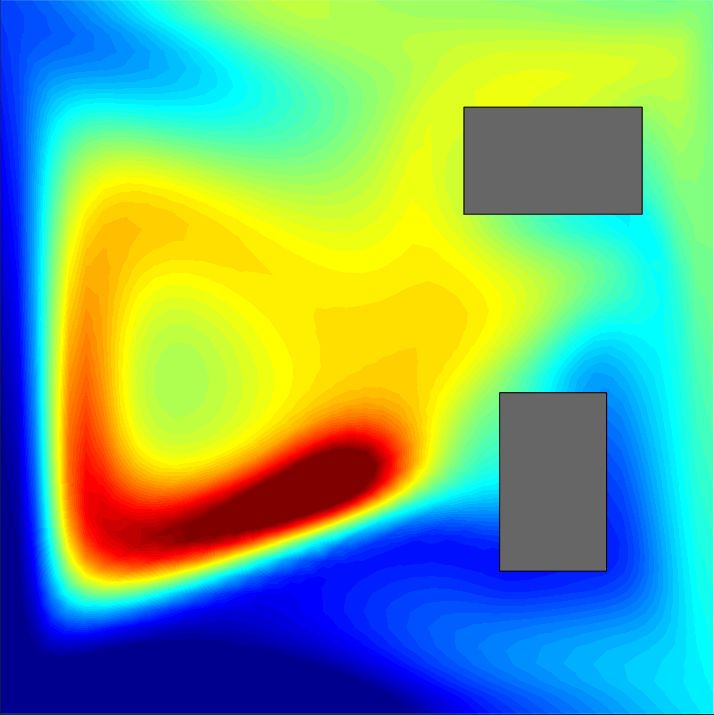}};
    \node (6) at (1.0*\pos,1.0*\pos-0.2*\pos){\includegraphics[width=.24\textwidth]{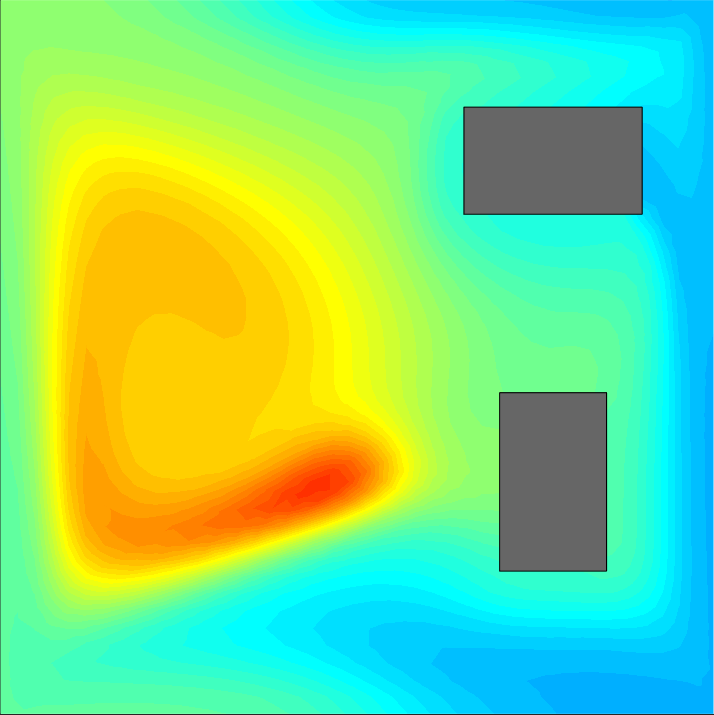}};
    \node (7) at (0*\pos-0.53*\pos, 0.5*\pos-0.2*\pos){\includegraphics[width=.24\textwidth]{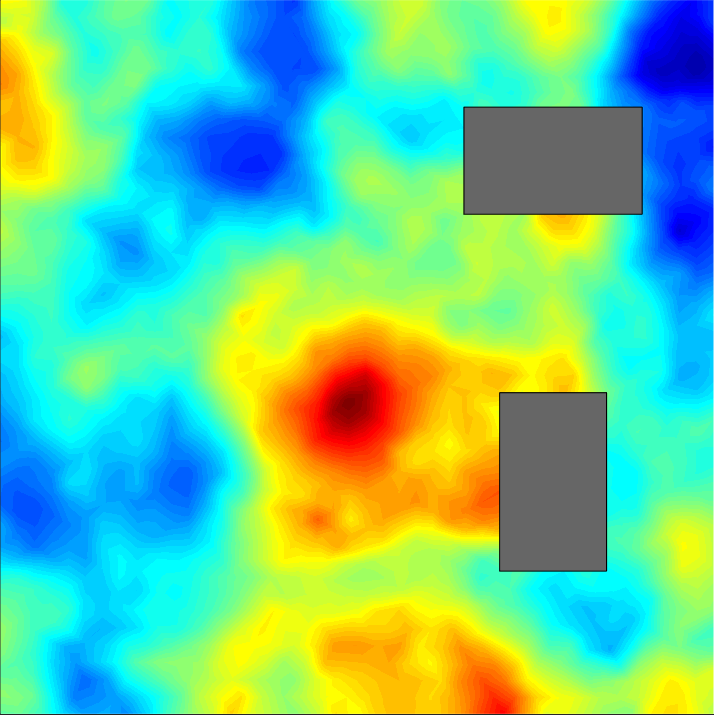}};
    \node (8) at (0.0*\pos, 0.5*\pos-0.2*\pos){\includegraphics[width=.24\textwidth]{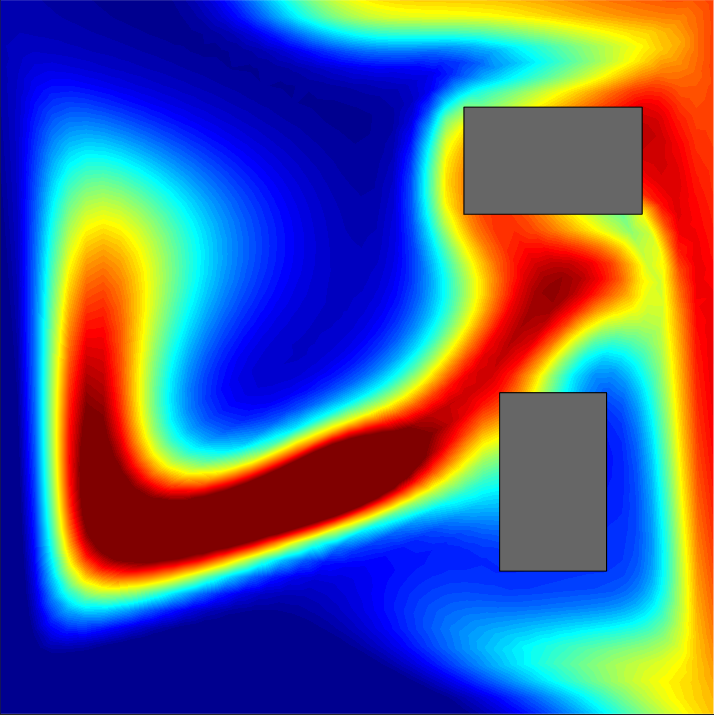}};
    \node (9) at (0.5*\pos, 0.5*\pos-0.2*\pos){\includegraphics[width=.24\textwidth]{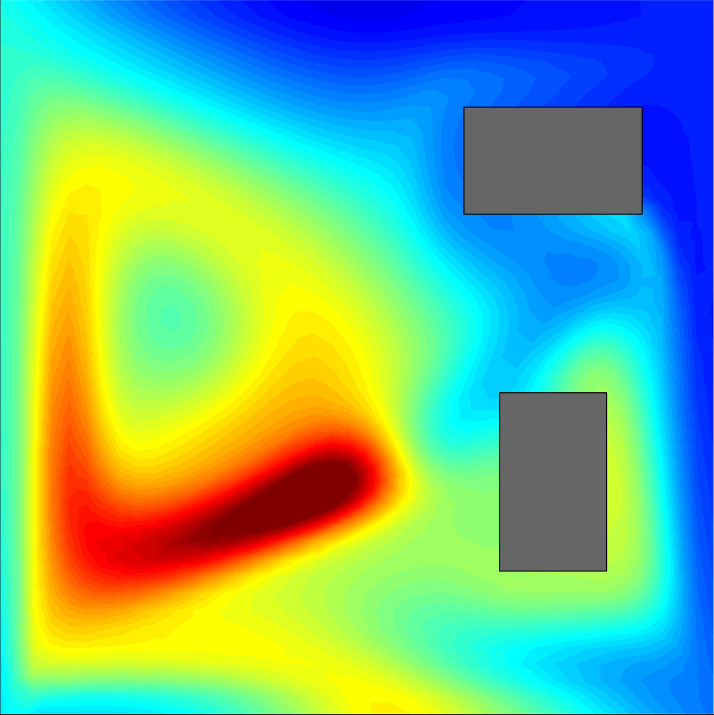}};
    \node (10) at (1.0*\pos, 0.5*\pos-0.2*\pos){\includegraphics[width=.24\textwidth]{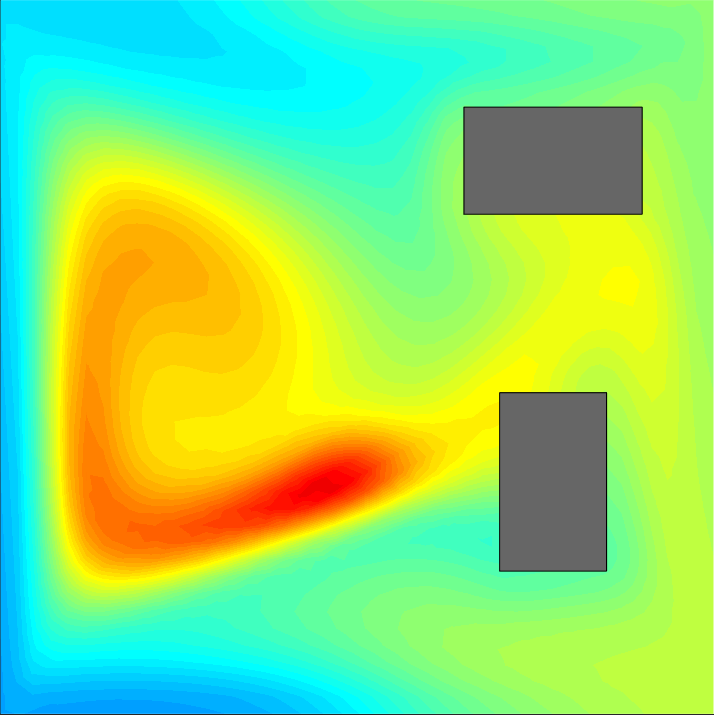}};
  \end{tikzpicture}
\caption{Shown in each row are snapshots of the concentration at times
  $t=0,0.4,0.6,1$ (from left to right). For the primary parameter
  $m$ entering on the right hand side of \cref{eq:ad1}, the ``truth''
  parameter shown in \cref{fig:velocity_prior}~(right) is used. For the secondary
  parameter $\iparb$, i.e., the initial condition, two different realizations
  from the distribution of $\iparb$ are used. Note that a different
  colorbar is used for the initial conditions than for the other snapshots.} 
\label{fig:forward_sol}
\end{figure}
\addtolength\abovecaptionskip{15pt}

\section{Computational results}\label{sec:numerics}
In this section, we present numerical results
for the model problem
described in \cref{sec:model}.
In \cref{sec:comparison}, we compare
the performance of regularized $\ell_0$-sparsification and greedy
approaches for computing
\mOEDs. Then, in \cref{subsec:results:posterior,subsec:results:MAP}, we demonstrate the
importance of taking the additional model uncertainty into account for
computing sensor placements.

\subsection{Comparison of sparsification algorithms}\label{sec:comparison}
Here, we compare the two different approaches to obtain binary \mOEDs{}
discussed in \cref{sec:method}.
As discussed in~\cref{subsec:sparsity}, when using $\ell_0$-sparsification we
solve a sequence of optimization problems with non-convex penalty functions
using a gradient-based optimization
algorithm. Here, we use MATLAB's interior point quasi-Newton
solver provided by the \verb+fmincon+ function, which we supply with routines implementing
the \mOED~objective and its gradient.  In contrast, the greedy
approach only requires the \mOED~objective.  As can be seen in
\cref{fig:comparedesigns}~(left), the greedy and the $\ell_0$-sparsified
designs perform similary. While in this figure the
$\ell_0$-sparsification finds slightly lower objective values, we have
also observed tests where the objective values are identical or the
greedy approach is slightly better.

It is also important to
consider the computational cost of these algorithms.  We do so by recording the
number of \mOED{} objective function evaluations required by the two algorithms in
\cref{fig:comparedesigns}~(right). Note that the cost of greedy
sensor placement scales with the number of sensors in the optimal
design, see also~\cref{equ:greedy_cost}. The cost of
the $\ell_0$-sparsification, in terms of function evaluations, remains
nearly constant. Of course, the regularized $\ell_0$-sparsification
method requires gradients additionally to objective
evaluations. However, as discussed in \cref{subsec:oedobs-and-grad},
the additional cost of computing the gradient is small compared to the
cost of \mOED~objective function evaluation. Therefore, the number of 
objective function evaluations is a reasonable  measure to compare 
the cost of the two algorithms.

\begin{figure}[ht]
  \centering
  \includegraphics[width=.8\columnwidth]{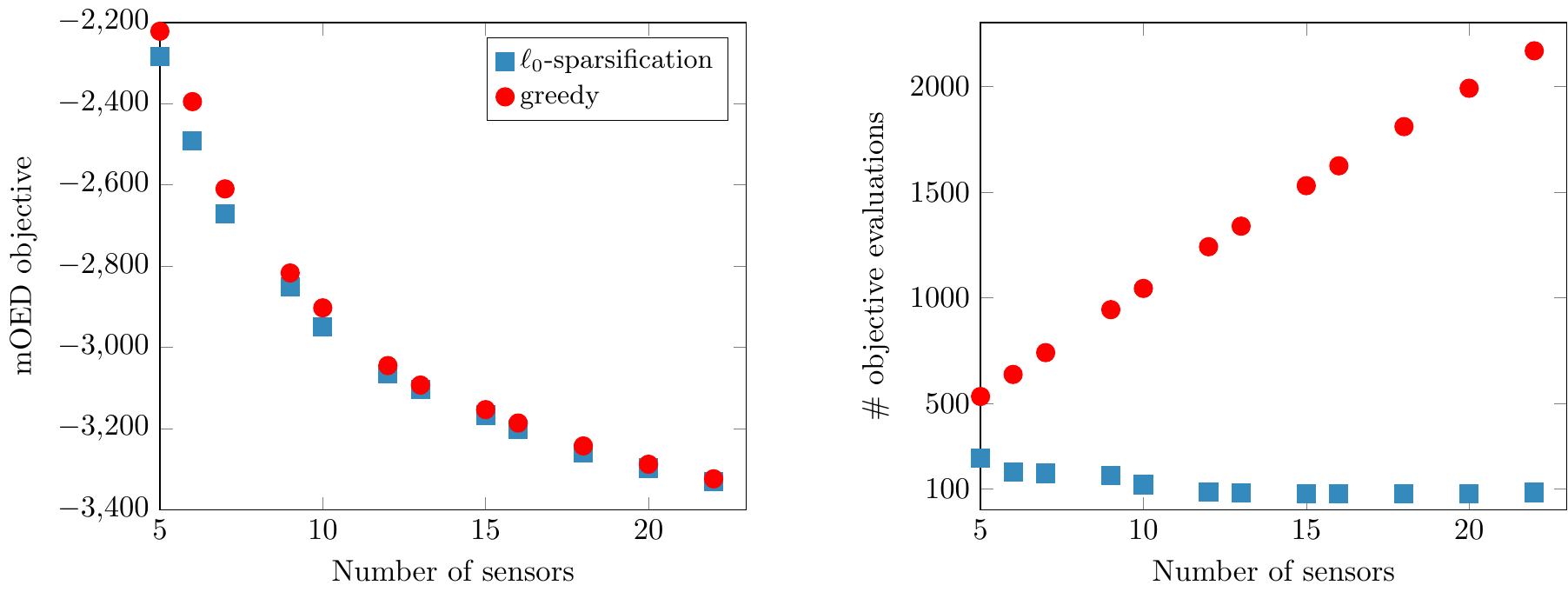} \hspace{.8cm}
  \caption{Left: \mOED{} objective values ($y$-axis) plotted
    against number of sensors ($x$-axis) for
    the greedy (red dots) and the
    $\ell_0$-sparsification approaches (blue dots). Right:
    Number of \mOED{} objective evaluations required to converge for
    computing greedy (red)
    and $\ell_0$-sparsified (blue) designs.}
  \label{fig:comparedesigns}
\end{figure}

In the remainder of this section, where we compare the performance of designs
obtained with and without marginalization, we use the greedy approach to find
optimal designs. This is motivated by the fact that the greedy approach facilitates
computing (near) optimal designs with a desired number of sensors,
while the $\ell_0$-sparsification approach only provides indirect
control on the number of sensors by changing the penalty parameter $\gamma$.

\subsection{Studying the posterior uncertainty}\label{subsec:results:posterior}
Next, we compare the performance of designs obtained by performing mOED against
those using OED with no marginalization in terms of the resulting marginal
posterior uncertainty. Note that designs obtained without marginalization,
which we simply refer to as OED, minimize the classical A-optimality criterion
$\psi$ in \cref{equ:classical_Aoptimal} whereas designs with marginalization
minimize the \mOED{} criterion in \cref{equ:obj}.

\def \pos {0.5\columnwidth}
\addtolength\abovecaptionskip{-15pt}
\begin{figure}[ht]\centering
  \label{fig:designs-and-marginal-post-stdev}
  \begin{tikzpicture}
   \node (1) at (0*\pos-0.3*\pos, 0.0*\pos){\includegraphics[width=.28\textwidth]{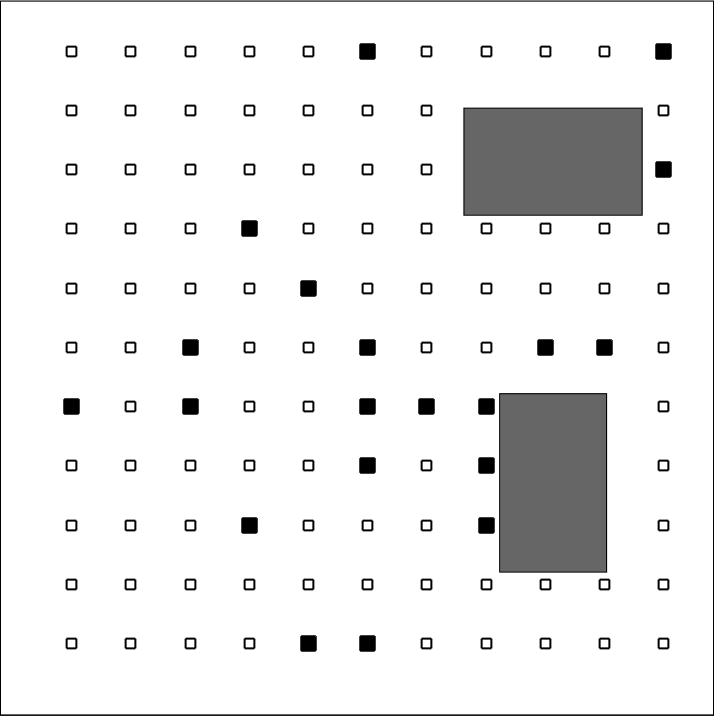}};
   \node (1) at (0.33*\pos, 0.0*\pos){\includegraphics[width=.28\textwidth]{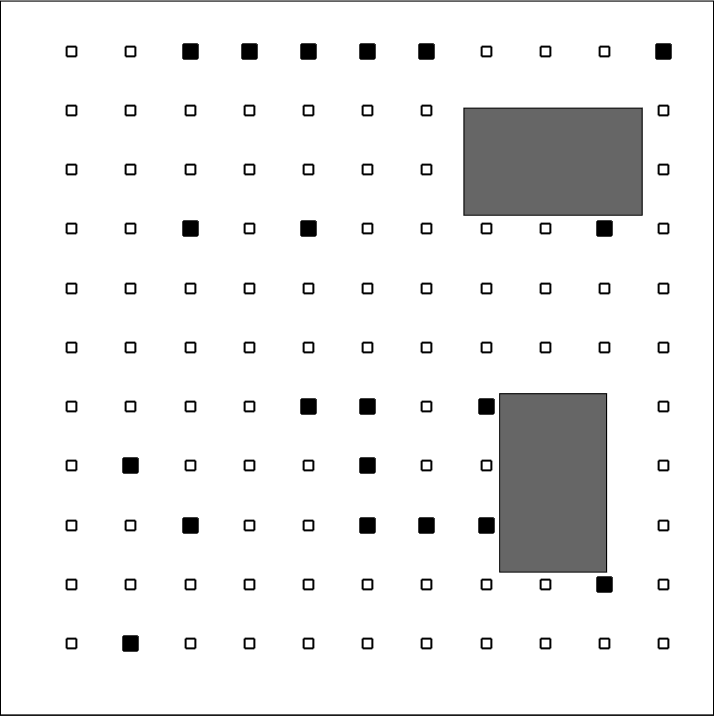}};
   \node (1) at (1.03*\pos, 0.0*\pos-0.02*\pos){\includegraphics[width=.36\textwidth]{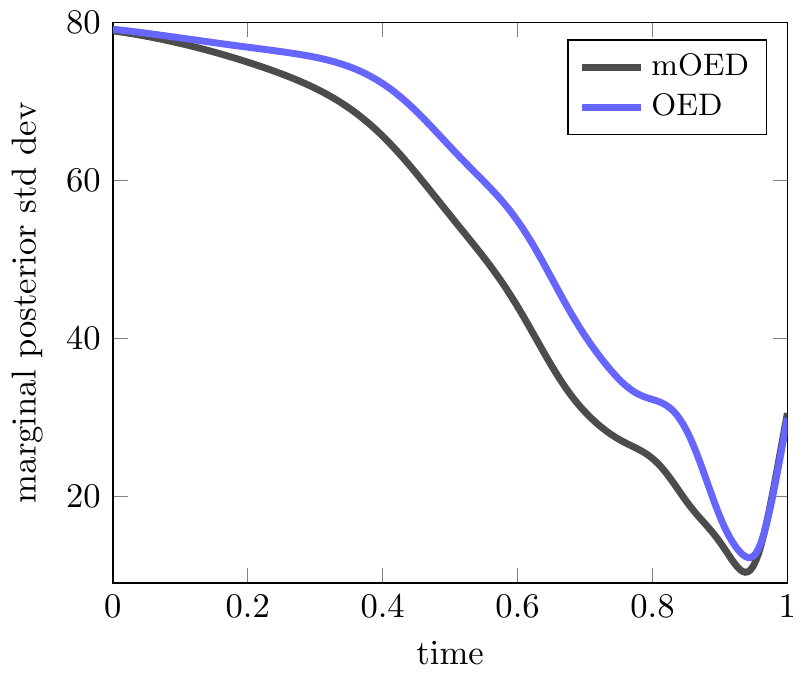}};
  \end{tikzpicture}
\caption{Shown are A-optimal designs with 20 sensors (filled
  squares) using \mOED{} (left) and OED without marginalization
  (center), i.e., the design obtained with OED neglecting secondary
  uncertainties. Inactive sensors are shown as empty
  squares. On the right, the marginal posterior standard deviation
  field (i.e., square root of the diagonal of $\postcovm(\vec{w})$
  in~\cref{eq:post_mb_blocks}) is shown for the two designs.
  \label{fig:designs}}
\end{figure}
\addtolength\abovecaptionskip{15pt}

\Cref{fig:designs} shows two designs with 20 sensors, one taking into
account the secondary uncertainty through marginalization, and one
assuming that there is no secondary uncertainty.
On the right panel of \cref{fig:designs}, the
pointwise standard deviation of the marginalized posterior
distribution are shown for the two sensor placements. The following
conclusions can be drawn. First, note that \mOED{} is superior, with
respect to the marginalized posterior variance, to the design computed
without taking the secondary uncertainty into account. Of course, this
is by construction of \mOEDs. However, the difference is significant
and exists for all times $t\in \cal T$. Second, since measurements are
taken around the final time, the uncertainty is more reduced for later
times.
However, close to the final time $T$, the
uncertainty increases again as there is not enough time for the
concentration field
to propagate to and be picked up by sensors.

\subsection{Study of MAP points}\label{subsec:results:MAP}
Next, we compare MAP points computed with the
\mOED{} and OED designs shown in \cref{fig:designs}. Note that the MAP
point for \mOED{} does not depend on a realization of the secondary parameter (see
\cref{eq:post_mb_blocks}), while it does for OED without marginalization
(see \cref{equ:weighted_posterior_fixedb}). 
In \cref{fig:design-compare}, we show the MAP point for the \mOED,
which recovers features from the ``truth'' parameter but resorts to the prior mean when little
information can be gathered from observations. 

As mentioned above, we need a realization of the secondary parameter $\iparb$
when computing the MAP point using the classical OED.
If we knew the ``truth'' $\iparb$, the additional uncertainty would vanish and the problem
reduces to
an inverse (and OED) problem with fully specified model as, e.g., in
\cite{AlexanderianPetraStadlerEtAl14}. The
corresponding MAP point, shown in blue in \cref{fig:design-compare},
slightly improved %
compared to the MAP point from
the \mOED{} formulation. However, in general the ``truth''
secondary parameter is unknown, and we only know its distribution.
If random draws from the secondary
parameter distribution are used in the MAP computation, the
model error is  underestimated and
the corresponding MAP points may be
poor. This can be seen in \cref{fig:design-compare}, where MAP points
obtained  with random
draws from the distribution of $\iparb$ are shown in red.

\begin{figure}[ht]\centering
  \includegraphics[width=.7\textwidth]{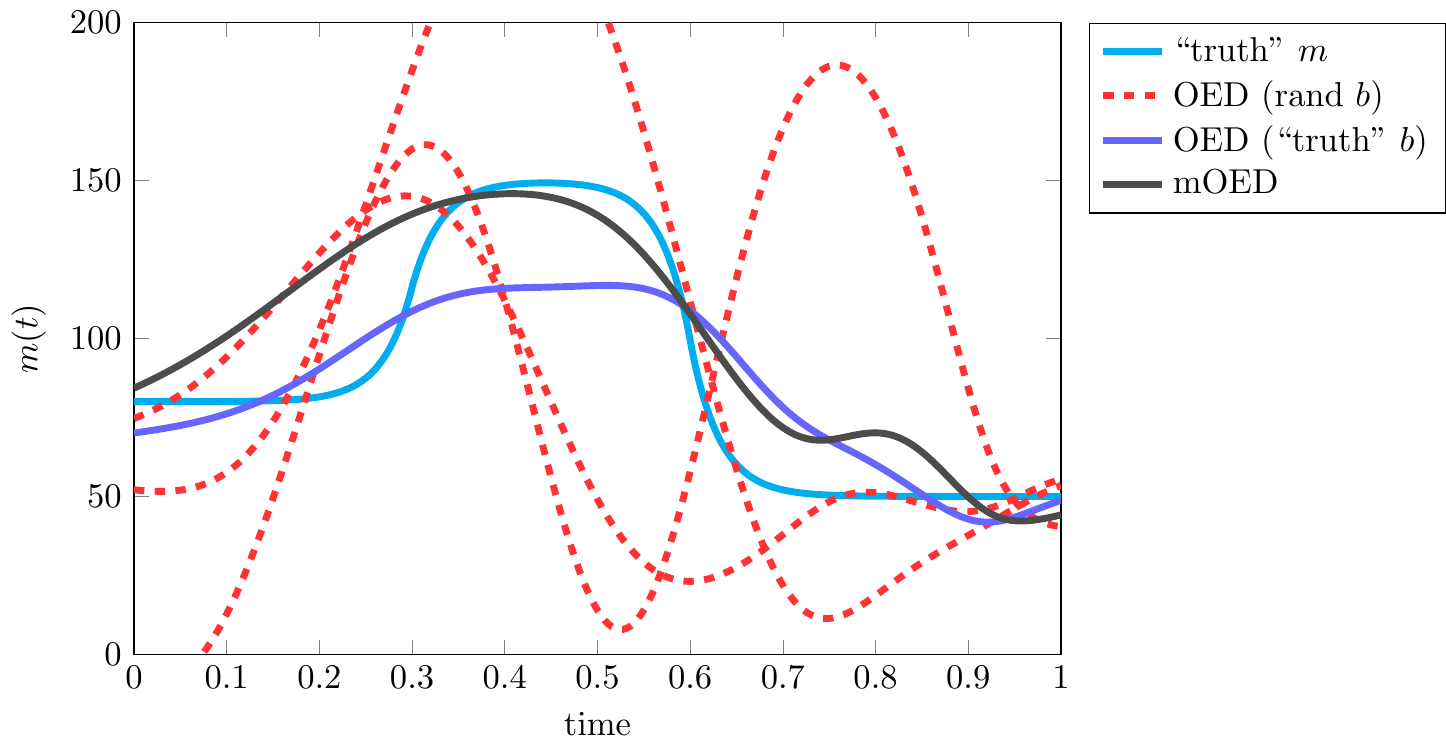}
  \caption{Comparison of MAP estimates computed with \mOED{} and OED
    without marginalization. Shown are the MAP estimates computed
    using sensor placements obtained via \mOED{} (black solid line),
    OED with the secondary parameter $\iparb$ set to the ``truth''
    (blue solid line), and OED with $\iparb$ taken as realizations
    from corresponding prior distribution (red dotted
    lines). \label{fig:design-compare}}
\end{figure}

The above discussed difference between \mOED{} and OED without
marginalization is summarized in
\cref{fig:design-compare-distribution}.  On the left, we plot the
relative $L^2(\T)$-error between the MAP point and the ``truth'' primary
parameter versus the \mOED{} objective. Using OED with
random draws for $\iparb$ result in MAP points that tend to be further
from the ``truth'' parameter than the \mOED{} MAP point. If
the ``truth'' secondary parameter is used in the computation of the MAP point using OED,
the reconstruction is slightly better than the result of \mOED{}.
It can also be seen that the \mOED{} objective is independent from
draws of the secondary parameter, as also discussed above.
The results in \cref{fig:design-compare-distribution}~(left) depend on 
the noise realizations in the synthetic data. In \cref{fig:design-compare-distribution}~(right),
we show the probability density function of the error
between the MAP point and the ``truth'' primary parameter for random
observation noise. As can be seen, 
it is slightly more likely to obtain a better MAP point when using OED
with the ``truth'' parameter than with \mOED{}. However,
it can clearly be seen that \mOED{} MAP  points significantly outperform
OED MAP points with random realizations from the prior
distribution of~$\iparb$.

\begin{figure}[ht]\centering
  \includegraphics[width=.9\textwidth]{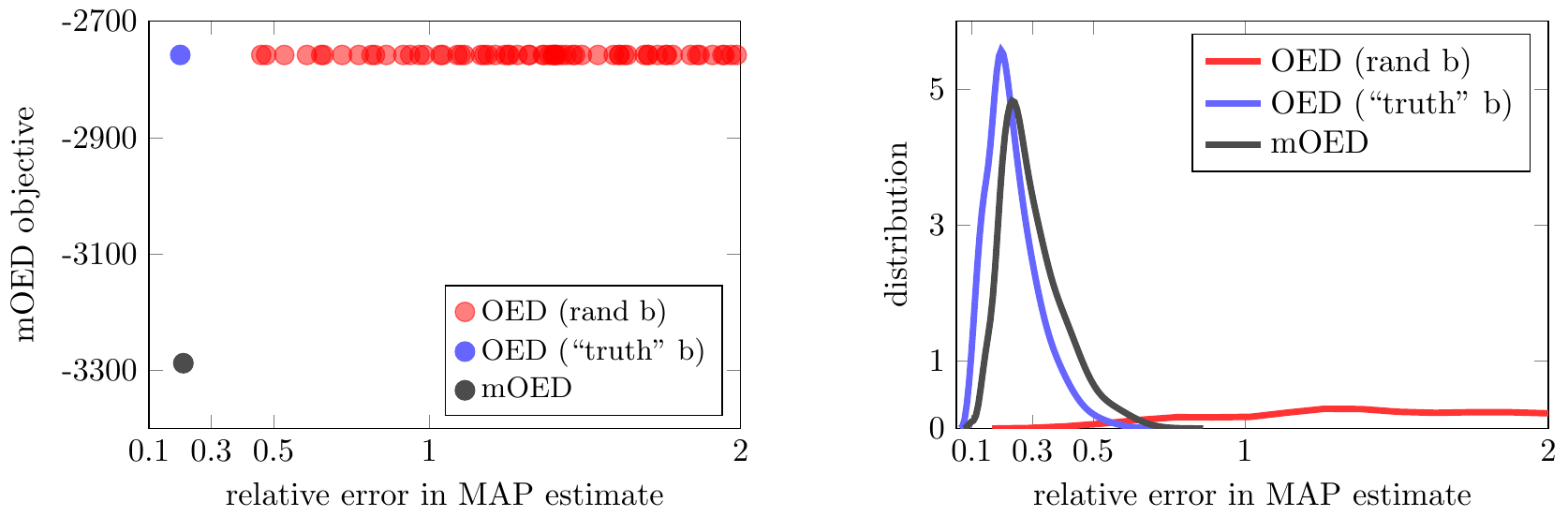}
  \caption{\emph{Left:} Relative error in the MAP estimate ($x$-axis)
    and reduction in the objective ($y$-axis) for \mOED{} (black dot),
    OED with the secondary parameter $\iparb$ set to the ``truth''
    (blue dot), and OED with $\iparb$ taken as different realizations
    of $\iparb$ (red dots). \emph{Right:} The distribution of the
    errors with various realizations of the noise in the
    data. Note that the $x$-axis is cut
    at 2 due to the long tail of the error distribution
    corresponding to OED with $\iparb$ taken as different realizations
    of $\iparb$. In this study, we used $200$ samples of the secondary parameter, 
    and $500$ samples of measurement noise.
    \label{fig:design-compare-distribution}}
\end{figure}

\section{Conclusion}\label{sec:conclusions}
In this article, we have considered linear inverse problems with reducible
model uncertainty and presented a mathematical and computational framework for
computing marginalized A-optimal sensors placements.  Our results show that it
is important to take into account additional sources of model uncertainty for
the optimal design and the inverse problem in general.
The designs computed by minimizing the marginalized
A-optimality criterion are superior compared to classical
A-optimal designs, in terms of the quality of the estimated primary parameters:
the marginalized optimal designs result in optimal uncertainty reduction as
well as more accurate MAP estimates. The overall conclusions support the claim
made in this article's title, namely that in the context of design of inverse
problems, it is good to know what you don't know. This information should be 
used when computing optimal designs.

An important direction for future work is design of nonlinear inverse problems
under model uncertainty.  A related direction is a sensitivity
analysis framework for detecting sources of model uncertainty that are most
important to the solution of the inverse problem.  This would enable
incorporating only the most important sources of model uncertainty in the OED
problem, hence reducing the computational complexity of the problem.
For deterministic inverse problems, first steps in
this direction are presented
in~\cite{SunseriHartVanBloemenWaandersAlexanderian20}.

\bibliographystyle{siamplain}
\bibliography{refs}

\begin{thebibliography}{10}

\bibitem{AkcelikBirosDraganescuEtAl05}
{\sc V.~Ak\c{c}elik, G.~Biros, A.~Dr\u{a}g\u{a}nescu, O.~Ghattas, J.~Hill, and
  B.~van Bloemen~Waanders}, {\em Dynamic data-driven inversion for terascale
  simulations: Real-time identification of airborne contaminants}, in
  Proceedings of SC2005, Seattle, 2005.

\bibitem{Alexanderian20}
{\sc A.~Alexanderian}, {\em Optimal experimental design for {B}ayesian inverse
  problems governed by {PDE}s: A review}, Preprint,  (2020).
\newblock \url{https://arxiv.org/abs/2005.12998}.

\bibitem{AlexanderianPetraStadlerEtAl14}
{\sc A.~Alexanderian, N.~Petra, G.~Stadler, and O.~Ghattas}, {\em A-optimal
  design of experiments for infinite-dimensional {B}ayesian linear inverse
  problems with regularized $\ell_0$-sparsification}, SIAM J. Sci. Comput., 36
  (2014), pp.~A2122--A2148.

\bibitem{Aravkin12VanLeeuwe12}
{\sc A.~Y. Aravkin and T.~Van~Leeuwen}, {\em Estimating nuisance parameters in
  inverse problems}, Inverse Problems, 28 (2012), p.~115016.

\bibitem{AtkinsonDonev92}
{\sc A.~C. Atkinson and A.~N. Donev}, {\em Optimum Experimental Designs},
  Oxford, 1992.

\bibitem{AttiaAlexanderianSaibaba18}
{\sc A.~Attia, A.~Alexanderian, and A.~K. Saibaba}, {\em Goal-oriented optimal
  design of experiments for large-scale {B}ayesian linear inverse problems},
  Inverse Problems, 34 (2018), p.~095009.

\bibitem{Bui-ThanhGhattasMartinEtAl13}
{\sc T.~Bui-Thanh, O.~Ghattas, J.~Martin, and G.~Stadler}, {\em A computational
  framework for infinite-dimensional {B}ayesian inverse problems. {P}art {I}:
  {T}he linearized case, with application to global seismic inversion}, SIAM J.
  Sci. Comput., 35 (2013), pp.~A2494--A2523.

\bibitem{ChalonerVerdinelli95}
{\sc K.~Chaloner and I.~Verdinelli}, {\em Bayesian experimental design: A
  review}, Statist. Sci., 10 (1995), pp.~273--304.

\bibitem{ChamonRibeiro17}
{\sc L.~Chamon and A.~Ribeiro}, {\em Approximate supermodularity bounds for
  experimental design}, in Advances in Neural Information Processing Systems,
  2017, pp.~5403--5412.

\bibitem{ConstantinescuBessacPetraEtAl20}
{\sc E.~M. Constantinescu, N.~Petra, J.~Bessac, and C.~G. Petra}, {\em
  Statistical treatment of inverse problems constrained by differential
  equations-based models with stochastic terms}, SIAM/ASA J. Uncertain.
  Quantif., 8 (2020), pp.~170--197.

\bibitem{DaPrato06}
{\sc G.~Da~Prato}, {\em An introduction to infinite-dimensional analysis},
  Springer Science \& Business Media, 2006.

\bibitem{DaonStadler18}
{\sc Y.~Daon and G.~Stadler}, {\em Mitigating the influence of boundary
  conditions on covariance operators derived from elliptic {PDEs}}, Inverse
  Probl. Imaging, 12 (2018), pp.~1083--1102.

\bibitem{FohringHaber16}
{\sc J.~Fohring and E.~Haber}, {\em Adaptive {A}-optimal experimental design
  for linear dynamical systems}, SIAM/ASA J. Uncertain. Quantif., 4 (2016),
  pp.~1138--1159.

\bibitem{HaberHoreshTenorio08}
{\sc E.~Haber, L.~Horesh, and L.~Tenorio}, {\em Numerical methods for
  experimental design of large-scale linear ill-posed inverse problems},
  Inverse Problems, 24 (2008), pp.~125--137.

\bibitem{HaberMagnantLuceroEtAl12}
{\sc E.~Haber, Z.~Magnant, C.~Lucero, and L.~Tenorio}, {\em Numerical methods
  for {A}-optimal designs with a sparsity constraint for ill-posed inverse
  problems}, Comput. Optim. Appl.,  (2012), pp.~1--22.

\bibitem{HandcockStein93}
{\sc M.~S. Handcock and M.~L. Stein}, {\em A {B}ayesian analysis of kriging},
  Technometrics, 35 (1993), pp.~403--410.

\bibitem{HermanAlexanderianSaibaba20}
{\sc E.~Herman, A.~Alexanderian, and A.~K. Saibaba}, {\em Randomization and
  reweighted $\ell_1$-minimization for {A}-optimal design of linear inverse
  problems}, SIAM J. Sci. Comput., accepted (2020).
\newblock \url{https://arxiv.org/abs/1906.03791}.

\bibitem{Jagalur-MohanMarzouk20}
{\sc J.~Jagalur-Mohan and Y.~Marzouk}, {\em Batch greedy maximization of
  non-submodular functions: Guarantees and applications to experimental
  design}, Preprint,  (2020).
\newblock \url{https://arxiv.org/abs/2006.04554}.

\bibitem{KaipioKolehmainen13}
{\sc J.~Kaipio and V.~Kolehmainen}, {\em Approximate marginalization over
  modeling errors and uncertainties in inverse problems}, Bayesian Theory and
  Applications,  (2013), pp.~644--672.

\bibitem{KaipioSomersalo05}
{\sc J.~Kaipio and E.~Somersalo}, {\em Statistical and Computational Inverse
  Problems}, vol.~160 of Applied Mathematical Sciences, Springer-Verlag, New
  York, 2005.

\bibitem{KolehmainenTarvainenArridgeEtAl11}
{\sc V.~Kolehmainen, T.~Tarvainen, S.~R. Arridge, and J.~P. Kaipio}, {\em
  Marginalization of uninteresting distributed parameters in inverse
  problems-application to diffuse optical tomography}, Int. J. Uncertain.
  Quantif., 1 (2011).

\bibitem{KovalAlexanderianStadler20}
{\sc K.~Koval, A.~Alexanderian, and G.~Stadler}, {\em Optimal experimental
  design under irreducible uncertainty for linear inverse problems governed by
  {PDEs}}, Inverse Problems, accepted (2020).
\newblock \url{https://arxiv.org/abs/1912.08915}.

\bibitem{KrauseSinghGuestrin08}
{\sc A.~Krause, A.~Singh, and C.~Guestrin}, {\em Near-optimal sensor placements
  in {G}aussian processes: Theory, efficient algorithms and empirical studies},
  J. Mach. Learn. Res., 9 (2008), pp.~235--284.

\bibitem{LindgrenRueLindstrom11}
{\sc F.~Lindgren, H.~Rue, and J.~Lindstr{\"o}m}, {\em An explicit link between
  {G}aussian fields and {G}aussian {M}arkov random fields: the stochastic
  partial differential equation approach}, J. R. Stat. Soc. Ser. B. Stat.
  Methodol., 73 (2011), pp.~423--498.

\bibitem{LuShiou02}
{\sc T.-T. Lu and S.-H. Shiou}, {\em Inverses of 2$\times 2$ block matrices},
  Comput. Math. Appl., 43 (2002), pp.~119--129.

\bibitem{Nagel17}
{\sc J.~B. Nagel}, {\em Bayesian techniques for inverse uncertainty
  quantification}, PhD thesis, ETH Zurich, 2017.

\bibitem{NicholsonPetraKaipio18}
{\sc R.~Nicholson, N.~Petra, and J.~P. Kaipio}, {\em Estimation of the {R}obin
  coefficient field in a {P}oisson problem with uncertain conductivity field},
  Inverse Problems, 34 (2018), p.~115005.

\bibitem{Ortega87}
{\sc J.~M. Ortega}, {\em Matrix theory: a second course}, The University Series
  in Mathematics, Plenum Press, New York, 1987.

\bibitem{RoininenHuttunenLasanen14}
{\sc L.~Roininen, J.~M. Huttunen, and S.~Lasanen}, {\em Whittle-{M}at{\'e}rn
  priors for {B}ayesian statistical inversion with applications in electrical
  impedance tomography}, Inverse Probl. Imaging, 8 (2014), p.~561.

\bibitem{RuthottoChungChung18}
{\sc L.~Ruthotto, J.~Chung, and M.~Chung}, {\em Optimal experimental design for
  inverse problems with state constraints}, SIAM J. Sci. Comput., 40 (2018),
  pp.~B1080--B1100.

\bibitem{ShulkindHoreshAvron18}
{\sc G.~Shulkind, L.~Horesh, and H.~Avron}, {\em Experimental design for
  nonparametric correction of misspecified dynamical models}, SIAM/ASA J.
  Uncertain. Quantif., 6 (2018), pp.~880--906.

\bibitem{Smith13}
{\sc R.~C. Smith}, {\em Uncertainty quantification: {T}heory, implementation,
  and applications}, vol.~12 of Computational Science and Engineering Series,
  SIAM, 2013.

\bibitem{Stuart10}
{\sc A.~M. Stuart}, {\em Inverse problems: {A B}ayesian perspective}, Acta
  Numer., 19 (2010), pp.~451--559.

\bibitem{SunseriHartVanBloemenWaandersAlexanderian20}
{\sc I.~Sunseri, J.~Hart, B.~van Bloemen~Waanders, and A.~Alexanderian}, {\em
  Hyper-differential sensitivity analysis for inverse problems constrained by
  partial differential equations}, Preprint,  (2020).
\newblock https://arxiv.org/abs/2003.00978.

\bibitem{Tong12}
{\sc Y.~L. Tong}, {\em The multivariate normal distribution}, Springer Science
  \& Business Media, 2012.

\bibitem{Troltzsch10}
{\sc F.~Tr\"oltzsch}, {\em Optimal Control of Partial Differential Equations:
  Theory, Methods and Applications}, vol.~112 of Graduate Studies in
  Mathematics, American Mathematical Society, 2010.

\bibitem{Ucinski05}
{\sc D.~Uci{\'n}ski}, {\em Optimal measurement methods for distributed
  parameter system identification}, CRC Press, Boca Raton, 2005.

\bibitem{WilliamsRasmussen06}
{\sc C.~K. Williams and C.~E. Rasmussen}, {\em Gaussian processes for machine
  learning}, vol.~2, MIT press Cambridge, MA, 2006.

\bibitem{Williams91}
{\sc D.~Williams}, {\em Probability with martingales}, Cambridge university
  press, 1991.

\end{thebibliography}

\end{document}